\documentclass{siamart190516}

\usepackage{amssymb,url,amsmath,graphicx,mathdots,bm}
\usepackage{accents}
\usepackage{tikz}
\usetikzlibrary{positioning, calc, chains}
\usetikzlibrary{shapes.misc}
\usetikzlibrary{shapes.symbols}
\usetikzlibrary{shapes.geometric}
\usetikzlibrary{shapes.arrows}
\usetikzlibrary{fit}
\usetikzlibrary{shadows}

\usepackage{pgfplots}
\usepackage{pgfplotstable}

\usepackage{subcaption}

\usepackage{algorithmic}

\newcommand{\bfx}{\mathbf{x}}

\newcommand{\hdiv}{H(\text{div})}

\newcommand{\mcm}{\mathcal{M}}
\newcommand{\mca}{\mathcal{A}}
\newcommand{\mcb}{\mathcal{B}}
\newcommand{\mcc}{\mathcal{C}}
\newcommand{\mcd}{\mathcal{D}}
\newcommand{\mcl}{\mathcal{L}}
\newcommand{\mcs}{\mathcal{S}}

\def\MM#1{\boldsymbol{#1}}
\newcommand{\pp}[2]{\tfrac{\partial #1}{\partial #2}}

\makeatletter
\renewcommand*\env@matrix[1][*\c@MaxMatrixCols c]{%
  \hskip -\arraycolsep
  \let\@ifnextchar\new@ifnextchar
  \array{#1}}
\makeatother

\title{Weighted-norm preconditioners for a multi-layer tide model}
\author{
  Colin J.~Cotter\thanks{Imperial College London, South Kensington Campus;
    London SW7 2AZ;
    Email: colin.cotter@imperial.ac.uk.  Supported by EPSRC EP/R029423/1}
  \and
  Robert C.~Kirby\thanks{Department of Mathematics, Baylor
    University; 1410 S.~4$^\text{th}$ St.; Waco, Texas 76706; Email: robert\_kirby@baylor.edu.  Supported by NSF 1912653 and 1909176.}
  \and
  Hunter Morris\thanks{Department of Mathematics, Baylor
    University; 1410 S.~4$^\text{th}$ St.; Waco, Texas 76706; Email: h\_morris@baylor.edu}
}

\begin{document}

\maketitle

\begin{abstract}
  We derive a linearized rotating shallow water system modeling tides, which can be discretized by mixed finite elements.
  Unlike previous models, this model allows for multiple layers stratified by density.
  Like the single-layer case~\cite{kirby2021preconditioning} a weighted-norm preconditioner gives a (nearly) parameter-robust method for solving the resulting linear system at each time step, but the all-to-all coupling between the layers in the model poses a significant challenge to efficiency.
  Neglecting the inter-layer coupling gives a preconditioner that degrades rapidly as the number of layers increases.
  By a careful analysis of the matrix that couples the layers, we derive a robust method that requires solving a reformulated system that only involves coupling between adjacent layers.
  Numerical results obtained using Firedrake~\cite{Rathgeber:2016} confirm the theory.
\end{abstract}


\section{Introduction}
Accurate modeling of tides plays a critical role in computational geosciences.  Tide models help geologists to understand sediment transport and coastal flooding, and they help oceanographers to study mechanisms for global circulation~\cite{GaKu2007,MuWu1998}.  Finite element methods offer theoretically and computationally robust and efficient discretizations of these methods, and are especially attractive in handling irregular coastlines or topography~\cite{We_etal2010}.  The literature contains many papers~\cite{CoLaReLe2010,CoHa2011,Ro2005,RoRoPo2007,RoBe2009,RoRo2008} studying mixed finite element pairs for discretization of layers of ocean and atmosphere models.  Much of this work relates to dispersion relations and conservation principles, although our work in~\cite{cotter2018mixed,CoKi} focuses on semidiscrete energy estimates related to the damping and corresponding error analysis, including a very broad class of possible nonlinear damping models.

This past work has focused on single-layer tide models derived under a linearization of the shallow water approximation.  Oceans tend to stratify by density according to depth, however, and more involved models can include multiple layers, each of which have different densities and are coupled together via hydrostatic pressure.  A derivation of the fully nonlinear multi-layer depth-averaged equations can be found in~\cite{audusse2005multilayer, mandli2011finite}.  Among many interesting features, these equations can lose hyperbolicity in situations approaching Kelvin-Helmholtz instability~\cite{le2003singularity}.
A further generalization with the number of layers varying spatially appears in~\cite{bonaventura2018multilayer}.
Here, we consider only a linearized model, suitable for tides rather than more general coastal flows, that does not have this difficulty.
We propose a mixed finite element discretization of this linearized multilayer model and develop effective preconditioners along the lines of those given in the single-layer case in~\cite{kirby2021preconditioning}.  The all-to-all coupling of the layers presents computational challenges, and special structure of the coupling matrix turns out to be critical.
We consider systems of equations arising from \emph{implicit} time stepping rather than the explicit methods in~\cite{audusse2005multilayer, bonaventura2018multilayer, mandli2011finite}.
These methods are better-suited for energy conservation in the absence of damping, and can allow large time steps with better stability, but the linear systems required at each time step are quite challenging  to solve.

\section{Model and discretization}
We consider a series of layers of fluid inhabiting a domain $\Omega \in \mathbb{R}^2$, with the top layer having
thickness $D_1$, the next layer $D_2$, and so on until $D_N$ for the
bottom layer with bottom boundary
$z=b(x,y)$.  The density of each layer is denoted by $\rho_i$, and we assume that $\rho_i < \rho_{i+1}$ -- the densities strictly increase between layers down a column.  
Typically, ocean water density varies between 1.02 and 1.07 $\mathrm{g}/\mathrm{cm}^3$.  Hence, we think of the change in density between the top and bottom layers as being small, and hence the density difference between two layers as quite small compared to $\rho_1$.  As a technical assumption that easily covers this case, we posit that
\begin{equation}
\label{eq:rhonnotbig}
\rho_N \leq 2 \rho_1.
\end{equation}

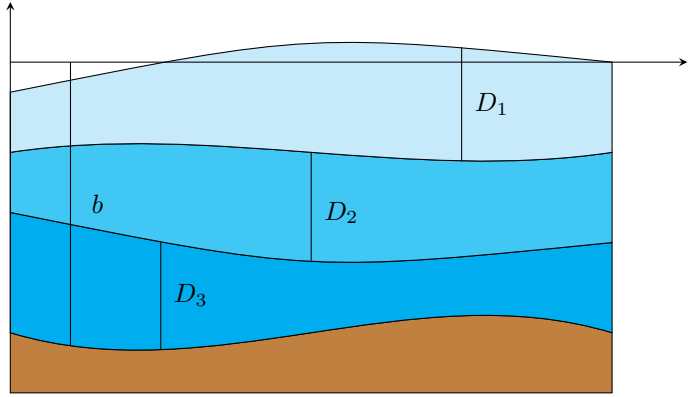
\begin{figure}
    \centering
    \begin{tikzpicture}[scale=2.0]
        \draw[fill=brown] (0, -0.4) -- (4, -0.4) -- 
        (4, 0) .. controls (2.67, 0.4) and (1.33, -0.4) .. (0, 0) -- cycle;
        \draw[fill=cyan] (0, 0) .. controls (1.33, -0.4) and (2.67, 0.4) .. (4, 0) -- 
        (4, 0.6) .. controls (2, 0.4) .. (0, 0.8) -- cycle;
        \draw[fill=cyan!60] (0, 0.8) .. controls (2, 0.4) .. (4, 0.6) --
        (4, 1.2) .. controls (2.67, 1.0) and (1.33, 1.4) .. (0, 1.2) -- cycle;
        \draw[fill=cyan!20] (0, 1.2) .. controls (1.33, 1.4) and (2.67, 1.0) .. (4, 1.2) --
        (4, 1.8) .. controls (2, 2.0) .. (0, 1.6) -- cycle;
        \draw[-stealth] (0, 1.8) -- (4.5, 1.8);
        \draw[-stealth] (0, 0) -- (0, 2.2);
        \draw[] (1, -0.1125) -- (1, 0.6);
        \node[] at (1.2, 0.25) {$D_3$};
        \draw[] (2, 0.48) -- (2, 1.2);
        \node[] at (2.2, 0.8) {$D_2$};
        \draw[] (3, 1.14) -- (3, 1.9);
        \node[] at (3.2, 1.52) {$D_1$};
        \draw[] (0.4, 1.8) -- (0.4, -0.086);
        \node[] at (0.58, 0.857) {$b$};
    \end{tikzpicture}
    \caption{Example of a one-dimensional fluid with three layers.}
    \label{fig:my_label}
\end{figure}
The horizontal fluid velocity within each layer will be denoted by
horizontal velocity $\MM{u}_i$.  We have acceleration due to gravity $g$,
and in this work we take the Coriolis parameter $f$ to be a constant less than 1.

We assume that the pressure is hydrostatic, meaning that the pressure
in each layer $i$ satisfies
\begin{equation}
  \pp{p}{z}|_i = -\rho_i g,
\end{equation}
so $p|_i = -\rho_i gz + c_i$ in each layer.

Using $p=0$ at the top surface, we have
\begin{equation}
  p|_1 = \rho_1 g\left(\sum_{j=1}^N D_j + b - z\right).
\end{equation}
Evaluating this at the bottom of the top layer gives
\begin{equation}
p|_1(z=\sum_{j=2}^N D_j + b)=\rho_1 g D_1.
\end{equation}
Then,
\begin{equation}
  \begin{split}
  	p|_2 &= \rho_2 g \left(\sum_{j=2}^N D_j + b - z + \frac{\rho_1}{\rho_2} \ D_1\right)= \rho_2 g \left(\sum_{j=1}^N D_j + b - z + \frac{\rho_1 - \rho_2}{\rho_2} \ D_1\right).
  \end{split}
\end{equation}
By induction or pattern-matching, we have
\begin{equation}
  \begin{split}
  p|_i &= \rho_i g \left(\sum_{j=i}^N D_j + b - z + \sum_{j=1}^{i-1}\frac{\rho_j}
  {\rho_i} D_j\right)\\
  &= \rho_i g \left(\sum_{j=1}^N D_j + b - z + \sum_{j=1}^{i-1}\frac{\rho_j-\rho_i}{\rho_i}
  D_j\right),
  \end{split}
\end{equation}

Under the assumption that the motion is columnar (that is, horizontal
velocity is independent of $z$) within each layer, the horizontal component of
the momentum equation becomes (after dividing by $\rho_i)$
\begin{equation}
  \begin{split}
	 \pp{\MM{u}_i}{t} + \MM{u_i}\cdot\nabla \MM{u}_i + f\MM{u}_i^\perp
	 = &-g\nabla \left(\sum_{j=1}^N D_j + \sum_{j=1}^{i-1} \frac{\rho_j-\rho_i}{\rho_i}
	 D_j + b\right) \\
	 &- \frac{C_i(\MM{u}_i)}{D_i} + \frac{F(t)}{D_i},
  \end{split}
\end{equation}
where we added a parameterization for bottom drag, with $C_i(\MM u_i))$ the
damping function, and $F(t)$ is the barotropic tidal forcing. The rationale
for the scaling with $D_i$ is that the drag is due to turbulence
assumed to occur in the bottom layer only. This turbulent flow exerts
an effective damping force proportional to the velocity in the bottom
layer, so the depth averaged momentum source is $\int_b^{b+D_N} F(u)
d z$, and then we divide by $D_i$ to get the equation for
$\MM{u}$.

Sometimes a simplified model is used under the rigid lid assumption,
in which we assume that $\sum_{j=1}^ND_j+b$
is constant. This is relevant because typically $(\rho_j-\rho_i)/\rho_i$ is
small, and so there are very fast ``barotropic'' waves where $\MM{u}_i$
is independent of $i$, and much slower ``baroclinic'' waves where the
free surface is more-or-less flat. It is the baroclinic tides that become
interesting since that is where tidally-generated energy is thought to be
dissipated as turbulence away from the bottom boundary.

Now, we nondimensionalize these equations as follows.
We introduce a characteristic vertical length scale $H$, horizontal length scale $L$, and velocity scale $V$.
We also introduce a reference density $\overline{\rho}$.
Then, we make the change of variables
\begin{equation}
	\bfx = \bfx^\prime H, \ \ \ t = \frac{H}{V} t^\prime
\end{equation}
so that
\begin{equation}
	\nabla = \tfrac{1}{H} \nabla^\prime, \ \ \
	\pp{}{t} = \tfrac{V}{H} \pp{}{t^\prime}.
\end{equation}
Then, we introduce dimensionless versions of our quantities as
\begin{equation}
	C_i = \tfrac{V}{L} C^\prime_{i}, \ \ \
	b = H b^\prime, \ \ \
	D_{i} = H D_{i}^\prime, \ \ \ 
	\MM{u}_{i} = V \MM{u}_{i}^\prime, \ \ \
	\rho_{i} = \overline{\rho} \rho^{\prime}_{i}.
\end{equation}
This gives the following non-dimensional equations:
\begin{align}
	\begin{split}
		\tfrac{V^2}{L}\left(\pp{\MM u^\prime_{i}}{t^\prime} + \MM{u}^{\prime}_{i}\cdot \nabla^{\prime}\MM{u}^{\prime}_{i}\right) + V f \MM{u^{\prime \perp}_{i}} = &-\frac{g H}{L}\nabla^{\prime}\left(\sum_{j=1}^{N} D^{\prime}_{j} + \sum_{j=1}^{i-1} \frac{\rho^{\prime}_{j} - \rho^{\prime}_{i}}{\rho^{\prime}_{i}} D^{\prime}_{j} + b^{\prime}\right) \\ &- \frac{V^2}{LD_{i}^\prime}\left(C^\prime_{i}(\MM{u}^{\prime}_i) - F^{\prime}(t^{\prime})\right),
	\end{split} \label{eq: nondim1}\\
	&\tfrac{H V}{L} \pp{D_{i}^{\prime}}{t^{\prime}} + \frac{H V}{L} \nabla^{\prime}\cdot(D_{i} \MM{u}^{\prime}_{i}) = 0, \label{eq: nondim2}
\end{align}
where $F^{\prime}(t^\prime) = \tfrac{L}{V^{2}H}F(\tfrac{L}{V}t^{\prime})$. Dropping the primes, dividing \eqref{eq: nondim1} by $\tfrac{V^{2}}{L}$, and dividing \eqref{eq: nondim2} by $\tfrac{HV}{L}$ produces:
\begin{align}
	\begin{split}
		\pp{\MM{u}_{i}}{t} + \MM{u}_{i}\cdot\nabla\MM{u}_{i} + \epsilon^{-1} \MM{u}^{\perp}_{i} = &-Fr^2 \nabla\left(\sum_{j=1}^{N} D_{j} + \sum_{j=1}^{i-1} \frac{\rho_{j} - \rho_{i}}{\rho_{i}} D_{j} + b\right) \\
	&- \frac{1}{D_{i}}\left(C_i(\MM{u}_{i}) - F(t)\right),
	\end{split}\\
	&\pp{D_{i}}{t} + \nabla\cdot(D_{i} \MM{u}_{i}) = 0,
\end{align}
where $Fr^2 = \tfrac{g H}{V^2}$ is the square of the Froude number, and $\epsilon^{-1} = \tfrac{fL}{V}$ is the reciprocal of the Rossby number.

The steady solutions are $\MM{u}_i=0$ $i=1,\ldots, N$, and
$D_i=\bar{D}_i=$constant for $i<N$, and $D_N - b = \bar{D}_N-b=$constant.  To
linearize, we write $D_i = \bar{D}_i+\eta_i$, where $\bar{D}_i$ is the
thickness of the layer when the system is at rest.  We assume
that $\eta_i$ and $\MM{u}_i$ are small, retaining only the linear terms
in the advection terms as well as replacing $D_i$ by $\bar{D}_i$ in the
forcing terms.  This gives
\begin{align}
  \pp{\MM{u}_i}{t} + \epsilon^{-1} \MM{u}_i^\perp
  = &-Fr^2 \nabla \left(\sum_{j=1}^N \eta_j + \sum_{j=1}^{i-1} \frac{\rho_j-\rho_i}{\rho_i}
  \eta_j\right) - \frac{1}{\bar{D}_{i}}\left(C_i(\MM{u}_{i}) - F(t)\right),\\
  &\pp{\eta_i}{t} + \nabla\cdot \left(\bar{D}_i\MM{u}_i\right) = 0.
\end{align}
Then we make the change of variables $\hat{\MM{u}}_i = \bar{D}_{i}\MM{u}_i$, which makes a kind of momentum rather than velocity the unknown field.
This gives:
\begin{align}
	\frac{1}{\bar{D}_i}\left(\pp{\hat{\MM{u}}_i}{t} + \epsilon^{-1} \hat{\MM{u}}_i^\perp\right)
	= &-Fr^2 \nabla \left(\sum_{j=1}^N \eta_j + \sum_{j=1}^{i-1} \frac{\rho_j-\rho_i}{\rho_i}
	\eta_j\right) - \frac{1}{\bar{D}_{i}}\left(\hat{C}_{i}(\hat{\MM{u}}_{i}) - F(t)\right), \label{eq: almostthere} \\
	&\pp{\eta_i}{t} + \nabla\cdot \hat{\MM{u}}_i = 0,
\end{align}
where $\hat{C}_i(\hat{\MM{u}}_{i}) = C_i\left(\tfrac{\hat{\MM{u}}_i}{\bar{D}_i}\right)$.  Although our model can be formulated with nonlinear damping as in~\cite{cotter2018mixed}, for the rest of the paper we will assume it is linear.

It will be convenient to multiply both sides of \eqref{eq: almostthere} by $\rho_i$. Carrying this out, and dropping the circumflexes, gives:
\begin{align}
	\mu_i\left(\pp{\MM{u}_i}{t} + \epsilon^{-1} \MM{u}_i^\perp\right)
	= -Fr^2 \nabla \left(\sum_{j=1}^{N}\mathcal{A}_{ij}
	\eta_i\right) - \mu_i\left(C_i(\MM{u}_i) - F(t)\right), \label{eq:scalar1}\\
	\pp{\eta_i}{t} + \nabla\cdot \MM{u}_i = 0, \label{eq:scalar2}
\end{align}
where
\begin{equation}
\label{eq:Adef}
\mathcal{A}_{ij} = \rho_{\min\{i, j\}}.
\end{equation}
For each layer, we let $\mu_i = \frac{\rho_i}{\bar{D}_i}$.

Let $\MM{u} = \begin{bmatrix}
	\MM{u}_1 \\
	\vdots \\
	\MM{u}_N
\end{bmatrix}$ and $\MM{\eta} = \begin{bmatrix}
\eta_1 \\
\vdots \\
\eta_N
\end{bmatrix}$.
Then we can write \eqref{eq:scalar1} and \eqref{eq:scalar2} in matrix-vector notation as follows:

\begin{align}
  \label{eq:mlvec}
  \mcm \pp{\MM{u}}{t} + \epsilon^{-1} \mcm \MM{u}^\perp + Fr^2 \nabla \left( \mca \MM{\eta}
  \right)
  + \mcb \MM{u} = F, \\
  \pp{\MM{\eta}}{t} + \nabla \cdot \MM{u} = 0,
\end{align}
where $\mcm$ is the diagonal matrix with $\mcm_{ii} = \mu_i$ and $\mcb$ is a semi-definite diagonal matrix modeling the damping.
Cases of particular interest include the fully definite case, as well as the case where $\mcb$ vanishes in all except the $N,N$ entry, which corresponds to damping only occuring in the bottom layer.

We let $W=L^2(\Omega)$ be the space of square-integrable functions over $\Omega$, with $W_0 = L^2_0(\Omega)$ the subspace of functions with vanishing mean.  
$V = H(\text{div}; \Omega)$ is the space of vector fields over $\Omega$ with square-integrable components and whose divergences are also square-integrable.
$V_0$ the subspace of functions with vanishing normal trace on $\partial \Omega$. 
We also let $W^N = W \times W \dots \times W$ consist of the $N$-way Cartesian product of the space with itself, with similar definitions of $W_0^N$, $V^N$, and $V^N_0$.  These spaces will be used to represent functions mapping $\Omega$ to the disturbances and velocities within each layer.

In the case of vanishing damping, one can apply standard energy techniques similar to wave equations, arrive at stability and well-posedness.
With damping present, one has a non-increasing energy so that we also expect such analysis to carry over.
However, our analysis of the single-layer case in~\cite{CoKi} gave long-time stability and error analysis for semidiscrete methods by showing the system energy is actually damped exponentially.
Similar results should follow readily if the damping is fully positive-definite, but determining the degree to which the results might hold in the semi-definite case is quite interesting.

Throughout, we let \( (\cdot, \cdot) \) denote the $(L^2)^N$ inner product, with
\begin{equation}
    (\MM u, \MM v) = \int_\Omega \sum_{i=1}^N \MM u_i \MM v_i dx.
\end{equation}
Without any subscript on the norm, we let $\|\MM u \| = \sqrt{(\MM u, \MM u)}$ be the standard $L^2$ norm.  For any smooth $\mcs$ mapping $\Omega$ into symmetric and uniformly positive-definite matrices, we also define the $\mcs$-weighted inner product by
\begin{equation}
    (\MM u, \MM v)_\mcs = (\mcs \MM u, \MM v).
\end{equation}
We assume that $\mcs$ is uniformly positive-definite over $\Omega$ so that
\begin{equation}
  \| \MM u \|_\mcs = \sqrt{(\MM u, \MM u)}_{\mcs}
\end{equation}
defines a norm equivalent to the $L^2$ norm with bounds
\begin{equation}
\label{eq:normequiv}
    C_\mcs \| \MM u \| \leq \| \MM u \|_{\mcs} \leq C^\mcs \| \MM u \|
\end{equation}
for some finite positive constants $C_\mcs$ and $C^{\mcs}$.

We also assume that $\mcb$ is bounded in the $L^2$ norm.  That is, there exists some $B^* < \infty$ such that for all $\MM u \in \MM V^N_h$,
\begin{equation}
\label{eq:bbd}
    \| \mcb \MM u \| \leq B^* \| \MM u \|.
\end{equation}

To arrive at a discrete model, we decompose $\Omega$ into a family of quasiuniform meshes $\{\mathcal{T}_h\}_h$ of triangles.  
For some fixed $k \geq 0$, we let $W_h \subset W$ be the space of all functions whose restriction to each $T \in \mathcal{T}_h$ are polynomials of degree $k$, and $V_h$ will consist of a suitable $\hdiv$ finite element space, such as the Raviart-Thomas element~\cite{RavTho77a} or Brezzi-Douglas-Marini~\cite{brezzi1985two} elements.  In the single-layer case, BDM elements may be preferable at small $\epsilon$ due to spurious modes appearing with RT~\cite{CoSh2012}. 
In particular, we assume that the property $\nabla \cdot V_h = W_h$ holds and that there exist suitable commuting projections~\cite{boffi2013mixed} that would enable stability and error analysis to hold.
Decomposition of $\Omega$ into quadrilateral meshes is also possible.  
If the mesh elements are not affine images of a reference square, some accuracy may be lost~\cite{arnold2005quad}. 

We let $V_h^N$ be the finite-dimensional space consisting of vectors of $N$ components, each in $V_h$, and $W_h^N$ with $N$ components in $W_h$.  By seeking a solution $\MM{u}: [0,T] \rightarrow V^N_h$ and $\MM{\eta}:[0,T]\rightarrow W^N_h$, a Galerkin discretization of~\eqref{eq:mlvec} is 
\begin{align}
  \label{eq:weak}
  \left(\pp{\MM{u}}{t}, \MM{v}\right)_\mcm
  +  \epsilon^{-1} \left(\MM{u}^\perp, \MM{v} \right)_\mcm
  - Fr^2\left(\MM{\eta}, \nabla \cdot \MM{v}  \right)_{\mca}
  + \left( \MM{u}, \MM{v} \right)_\mcb = \left(F, \MM v\right), \\
  \left( \pp{\MM{\eta}}{t}, \MM{w} \right)
  + \left( \nabla \cdot \MM{u}, \MM w\right) = 0,
\end{align}
for all $\MM v \in V_h^N$ and $\MM w \in W_h^N$.

To obtain a fully discrete method, we must specify some time-stepping scheme.  
For example, the implicit midpoint rule is symplectic and, in the damping-free case of $\mcb=0$, conserves the system energy exactly for this problem.  
We assume a constant step size $\Delta t$ and define discrete time levels $t_n = n \Delta t$.
Then, given initial conditions $\MM u_h^0$ and $\MM \eta_h^0$, the solution at each time level is approximated by
\begin{align}
  \label{eq:weak_ts}
  \begin{split}
  \left(\frac{\MM{u}^{n+1} - \MM u^n}{\Delta t}, \MM{v}\right)_{\mcm}
  + \epsilon^{-1}\left(\left(\MM{u}^{n+1/2}\right)^\perp, \MM{v} \right)_{\mcm} & \\
 - Fr^2\left(\MM{\eta}^{n+1/2}, \nabla \cdot \MM{v}  \right)_{\mca}
 + \left( \mcb\MM{u}^{n+1/2}, \MM{v} \right) = & \left(F^{n+1/2}, \MM v\right), \\
 \left( \frac{\MM{\eta}^{n+1} - \MM{\eta}^n}{\Delta t}, \MM{w} \right)
 + \left( \nabla \cdot \MM{u}^{n+1/2}, \MM w\right) = & 0,
 \end{split}
\end{align}
where we define $\MM u_h^{n+1/2} = \frac{1}{2} \left(\MM u_h^n + \MM u_h^{n+1} \right)$ and similarly for $\MM \eta_h^{n+1/2}$.
Multiplying through each equation by $\Delta t$ and moving known data to the right-hand side, we see that a variational problem of the form
\begin{equation}
\label{eq:canonicaleq}
    \begin{split}
        \left(\MM u , \MM v \right)_{\mcm}
        + \epsilon^{-1} k \left( \MM u^\perp, \MM v \right)_{\mcm}
        - Fr^2 k \left( \MM \eta , \nabla \cdot \MM v \right)_{\mca}
        + k \left( \mcb \MM u, \MM v \right) & = \left( F_1 , \MM v\right), \\
        \left( \MM \eta , \MM w \right)
        + k \left(\nabla \cdot \MM u , \MM w \right) 
        & = (F_2, \MM w)
    \end{split}
\end{equation}
must be solved at each time step, where $k>0$ is some small number related to the time step.  This equation is fairly generic -- other single-stage methods such as Crank-Nicolson or backward Euler give systems of the same form.  A multi-stage Runge-Kutta method, such as considered in~\cite{farrell2020irksome} for the wave equation, would give a more complicated system, although the diagonal blocks would have this form.

To simplify the analysis, we define the bilinear form
\begin{equation}
\label{eq:a}
\begin{split}
    a\left((\MM u , \MM \eta), (\MM v,  \MM w)\right) = & \left( \MM u , \MM v \right)_{\mcm}
        + \epsilon^{-1} k \left( \MM u^\perp, \MM v \right)_{\mcm}
        - Fr^2 k \left( \MM \eta , \nabla \cdot \MM v \right)_{\mca}
        + k \left( \mcb \MM u, \MM v \right) \\
        & + \left( \MM \eta , \MM w \right)
        + k \left( \nabla \cdot \MM u , \MM w \right),
\end{split}
\end{equation}
and the variational problem~\eqref{eq:canonicaleq} can then be compactly written as finding $(\MM u, \MM \eta) \in \MM V^N_h \times \MM W^N_h$ such that
\begin{equation}
    a\left((\MM u , \MM \eta), (\MM v,  \MM w)\right)
    = \left( F_1 , \MM v \right) + \left( F_2 , \MM w \right)
\end{equation}
for all $(\MM v, \MM w) \in \MM V^N_h \times \MM W^N_h$.

Now, we cast the discrete variational problem~\eqref{eq:canonicaleq} into matrix notation.
We let $\{\psi_i\}_{i=1}^{\dim V_h}$ be a basis for $V_h$.  Then, if $\MM e^j$ is the canonical basis vector in $\mathbb{R}^n$, 1 in entry $j$ and vanishing in other entries, functions of the form
\[
\MM \psi_i^j = \psi_i \MM e^j, \ \ \ 1 \leq i \leq \dim V_h, 1 \leq j \leq N
\]
form a basis for $\MM V_h^N$.  
Similarly, we let $\{ \phi_i \}_{i=1}^{\dim W_h}$ be a basis for $W_h$.  With
\[
\MM \phi_i^j = \phi_i \MM e^j, \ \ \ 1 \leq i \leq \dim V_h, 1 \leq j \leq N,
\]
and $\left\{ \left\{ \MM \phi_i^j \right\}_{i=1}^{\dim W_h} \right\}_{j=1}^N$ forms basis for $\MM W_h$.

In order to define matrices, we need to impose a total ordering on the basis functions for $\MM V_h^N$ and $\MM W_h^N$.  For example, for $1 \leq i \leq N \dim V_h$, we can write find unique $i_0$, $i_1$ such that $i = i_0 \dim V_h + i_1$ by integer division/remainder operations and then put
\[
\MM \Psi_i = \MM \psi_i^j = \psi_{i_1} \MM e^{i_0},
\]
with a similar total ordering for $\{ \Phi_i \}_{i=1}^{N\dim V_h}$.
This ordering imposes a block structure on the linear system by storing all the degrees of freedom within a layer contiguously. 

Before proceeding, give a remark on matrix notation, as several different kinds of matrices appear in this paper.  Matrices that act across the layers of the tide model, such as $\mca$, $\mcb$, and $\mcm$ have been denoted in calligraphic letters.  Discrete operators on a single layer, or equivalently, those discretizing a bilinear form over $V_h$ and/or $W_h$ will be denoted in italics.  To this end, we define:
\begin{equation}
\label{eq:singlelayermatrices}
    \begin{split}
        M^V_{ij} & = \left( \psi_j, \psi_i \right), \\
        M^{V,\kappa}_{ij} & = \left( \kappa \psi_j, \psi_i \right), \\
        \tilde{M}^{V} & = \left(\psi_j^\perp, \psi_i \right), \\
        M^{W}_{ij} & = \left( \phi_j, \phi_i \right), \\
        D_{ij} & = \left( \nabla \cdot \psi_j, \phi_i \right), \\
        E_{ij} & = \left( \nabla \cdot \psi_j, \nabla \cdot \psi_i \right).
    \end{split}
\end{equation}
Then, we use Roman block lettering to denote discrete operators over $\MM V_h^N$ and/or $\MM W_h^N$.  Such needed matrices are:
\begin{equation}
\label{eq:multilayermatrices}
    \begin{split}
        \mathrm{M}^{V}_{ij} & = \left( \MM \Psi_j, \Psi_i \right)_\mcm, \\
        \mathrm{M}^W_{ij} & = \left( \MM \Phi_j, \MM \Phi_i \right),  \\
        \tilde{\mathrm{M}}_{ij} & = \left( \MM \Psi_j^\perp, \MM \Psi_i \right)_\mcm, \\ 
        \mathrm{D}_{ij} & = \left( \nabla \cdot \MM \Psi_j , \MM \Phi_i \right), \\
        \mathrm{D}^{\mca}_{ij} & = \left(\nabla \cdot \MM \Psi_j , \MM \Phi_i \right)_{\mca}, \\
        \mathrm{B}_{ij} & = \left( \mcb \MM \Psi_j , \MM \Psi_i \right), \\
        \mathrm{E}_{ij} & = \left( \nabla \cdot \MM \Psi_j, \nabla \cdot \MM \Psi_i \right), \\ 
        \mathrm{E}^{\mca}_{ij} & = \left(\nabla \cdot \MM \Psi_j, \nabla \cdot \MM \Psi_i \right)_{\mca}.
    \end{split}
\end{equation}
Note that the matrices appearing in~\eqref{eq:multilayermatrices} have important substructure.  For example, we have that
\begin{equation}
\mathrm{M}^V = diag(M^{V,\mu_1}, \dots,  M^{V, \mu_N}).
\end{equation}
The first $N-1$ blocks are in fact constant coefficient and so equal to $\mu_i M^V$.  Due to the variable bathymetry, the bottom right block is not, but it is still symmetric and positive-definite.
The matrix $\mathrm{B}$ is also block diagonal and symmetric semi-definite.  
If the damping matrix $B$ is full-rank, it is definite.
Similarly, $\mathrm{M}^W$, $\tilde{\mathrm{M}}^V$, and $\mathrm{D}$, $\mathrm{E}$ are block diagonal.  In fact, $\mathrm{W}^W = I \otimes M^W$, $\mathrm{D} = I \otimes D$, and $\mathrm{E} = I \otimes E$, where $I$ is the $N \times N$ identity matrix.

The matrices $\mathrm{D}^{\mca}$ and $E^{\mca}$ also have structure, with
\begin{equation}
\begin{split}
    \mathrm{D}^{\mca} & = \mca \otimes D \\
\mathrm{E}^{\mca} & = \mca \otimes E
\end{split}
\end{equation}

A Galerkin discretization of~\eqref{eq:canonicaleq} then gives rise to a block matrix system of the form
\begin{equation}
\label{eq:blocksys}
    \begin{bmatrix}
        \mathrm{M}^V + \epsilon^{-1}k \tilde{\mathrm{M}}^V + k \mathrm{B} & -Fr^2k \left( \mathrm{D}^{\mca} \right)^T \\
        k \mathrm{D} & \mathrm{M}^W
    \end{bmatrix}
    \begin{bmatrix}
    \mathrm{u} \\ \mathrm{\eta}
    \end{bmatrix}    
    =\begin{bmatrix}
    \mathrm{F}_1 \\ \mathrm{F}_2
    \end{bmatrix}
\end{equation}

\section{A weighted-norm preconditioner}
Linear systems arising from finite element discretizations are typically solved using iterative methods such as the generalized minimum residual method (hence, GMRES)~\cite{saad1986gmres}.  
These methods have the advantage of requiring only matrix-vector products with the system matrix, but their performance depends strongly on the conditioning of the linear system.
The conditioning of the system matrix, and hence number of iterations required for convergence, can degrade as a function of mesh refinement and/or physical parameters.
In such cases, it is critical to \emph{precondition} the linear system by pre-multiplying a linear system
\[
A x = b
\]
by some linear operator $P^{-1}$ to obtain the equivalent system
\[
P^{-1} A x = P^{-1} b.
\]
One hopes to choose $P$ such that the iterative method converges much faster for $P^{-1} A$ than that of $A$ under the constraint that
the cost of applying $P^{-1}$ at each iteration not offset the gains obtained by reducing the iteration count.

When preconditioning finite element linear systems, it can be helpful to choose $P$ as discretizing some simple differential operator, such as an inner product on the underlying Hilbert space~\cite{kirby2010functional, mardal2011preconditioning}.
It is also frequently possible to incorporate physical parameters in the definition of the preconditioner in such a way as to minimize the dependence of the spectral bounds on those parameters.
We refer to these as ``weighted-norm'' preconditioners, and we adopt this perspective here.

In this section, we propose and analyze the matrix
\begin{equation}
	\label{eq:precond}
    \begin{bmatrix}
    \mathrm{M}^V + Fr^2 k^2 \mathrm{E}^{\mca} & 0 \\ 0 & \mathrm{M}^W
    \end{bmatrix}
\end{equation}
as a preconditioner for~\eqref{eq:blocksys}.
Because this matrix decouples the momentum and elevation variables, it should be far easier to invert than the original matrix.
The $\mathrm{M}^W$ block is itself quite simple, just a block diagonal matrix of mass matrices (which can be diagonal in the lowest order case).  However, the top left block couples all of the layer velocities together, and we take a closer look at this block in the following section.

This matrix arises from discretizing the bilinear form
\begin{equation}
\label{eq:b}
    b\left((\MM u, \MM \eta), (\MM v, \MM w) \right) =
    \left( \MM u, \MM v \right)_\mcm
    + Fr^2 k^2 \left( \nabla \cdot u , \nabla \cdot v \right)_{\mca} + \left( \MM \eta , \MM w \right)
\end{equation}
over $\MM V^N_h \times \MM W^N_h$.  This bilinear form is equivalent to the standard $\hdiv \times L^2$ inner product, with constants dependent upon the physical parameters.
We will prove norm equivalence by giving continuity and inf-sup bounds of the bilinear form $a$ in~\eqref{eq:a} with respect to the norm defined by the inner product $b$.

We first note that the matrix $\mathrm{D}^{\mca}$ appears in the first row of the system matrix, but $\mathrm{D}$ in the second.  Also, the two blocks are scaled differently with respect to the Froude number.
This structural asymmetry, complicates the analysis.  Rather than scaling the actual system to be solved, we can give analysis for an equivalent pair of bilinear forms.  To motivate this alternate pair, we 
rewrite the preconditioned matrix:
\begin{equation}
\label{eq:redopc}
\begin{split}
    & \begin{bmatrix}
    \mathrm{M}^V + Fr^2 k^2 \mathrm{E}^{\mca} & 0 \\ 0 & \mathrm{M}^W
    \end{bmatrix}^{-1}
    \begin{bmatrix}
        \mathrm{M}^V + \epsilon^{-1} k \tilde{\mathrm{M}}^V + k B & -Fr^2 k \left( \mathrm{D}^{\mca} \right)^T \\
        k \mathrm{D} & \mathrm{M}^W
    \end{bmatrix}
    \\
= & 
\begin{bmatrix}
    \mathrm{M}^V + Fr^2 k^2 \mathrm{E}^{\mca} & 0 \\ 0 & Fr^2 \mathrm{M}^{W,\mca}
\end{bmatrix}^{-1}
    \begin{bmatrix}
        \mathrm{M}^V + \epsilon^{-1} k \tilde{\mathrm{M}}^V + k B & -Fr^2 k \left( \mathrm{D}^{\mca} \right)^T \\
        Fr^2 k \mathrm{D}^{\mca} & Fr^2 \mathrm{M}^{W,\mca}
    \end{bmatrix},
\end{split}
\end{equation}
where we have inserted the identity, written as
\[
\begin{bmatrix}
I & 0 \\ 0 & Fr^2 \mca \otimes I
\end{bmatrix}^{-1}
\begin{bmatrix}
I & 0 \\ 0 & Fr^2 \mca \otimes I
\end{bmatrix}
\]
between the two matrices on the left-hand side.

The second matrix on the right-hand side discretizes of the bilinear form
\begin{equation}
\label{eq:ahat}
\begin{split}
    \hat{a}\left((\MM u , \MM \eta), (\MM v,  \MM w)\right) = & \left( \MM u , \MM v \right)_{\mcm}
        + \epsilon^{-1} k \left( \MM u^\perp, \MM v \right)_{\mcm}
        - Fr^2 k \left( \MM \eta , \nabla \cdot \MM v \right)_{\mca}
        + k \left( \mcb \MM u, \MM v \right) \\
        & + Fr^2 \left( \MM \eta , \MM w \right)_{\mca}
        + Fr^2 k \left( \nabla \cdot \MM u , \MM w \right)_{\mca},
\end{split}
\end{equation}
while the first discretizes the weighted inner product
\begin{equation}
\label{eq:bhat}
    \hat{b}\left((\MM u, \MM \eta), (\MM v, \MM w) \right) =
    \left( \MM u, \MM v \right)_\mcm
    + Fr^2 k^2 \left( \nabla \cdot \MM u , \nabla \cdot \MM v \right)_{\mca} + Fr^2 \left( \MM \eta , \MM w \right)_{\mca}.
\end{equation}
We further define the $\| \cdot \|_{\hat{b}}$ norm on $\MM V_h^N \times \MM W_h^N$ by
\begin{equation}
    \left\| \left( \MM u , \MM \eta \right) \right\|_{\hat{b}} = 
    \sqrt{\hat{b}\left( \left( \MM u , \MM \eta \right), \left( \MM u , \MM \eta \right) \right)}.
\end{equation}

Because of equality~\eqref{eq:redopc}, GMRES iteration for the matrix associated with bilinear form~\eqref{eq:a} preconditioned by that from~\eqref{eq:b} is exactly equivalent to that obtained from the matrices for~\eqref{eq:ahat} and~\eqref{eq:bhat}.  
We proceed to demonstrate norm equivalence for the latter pair.

\begin{theorem}
  \label{thm:upper}
For all $(\MM u, \MM \eta), (\MM v, \MM w)$ in $\MM V_h^N \times \MM W_h^N$, 
\begin{equation}
\hat{a} \left( \left(\MM u, \MM \eta\right), \left(\MM v, \MM w\right) \right)
\leq C \left\|\left(\MM u, \MM \eta\right) \right\|_{\hat{b}}
\left\| \left(\MM v, \MM w\right) \right\|_{\hat{b}},
\end{equation}
where
\begin{equation}
C = \max \left\{ 2, 1 + \tfrac{k}{\epsilon} + \frac{k B^*}{C_\mcm^2} \right\}.
\end{equation}
\end{theorem}
\begin{proof}
  Let $(\MM u, \MM \eta), (\MM v, \MM w) \in \MM V^N_h \times \MM W^N_h$ be given.  Then, applying the Cauchy-Schwarz inequality and noting $\cdot^\perp$ is pointwise an isometry, we have
  \begin{equation}
  \begin{split}
          \hat{a}\left((\MM u , \MM \eta), (\MM v,  \MM w)\right) = & \left( \MM u , \MM v \right)_{\mcm}
        + \epsilon^{-1} k \left( \MM u^\perp, \MM v \right)_{\mcm}
        - Fr^2 k \left( \MM \eta , \nabla \cdot \MM v \right)_{\mca}
        + k \left( \mcb \MM u, \MM v \right) \\
        & + Fr^2 \left( \MM \eta , \MM w \right)_{\mca}
        + Fr^2 k \left( \nabla \cdot \MM u , \MM w \right)_{\mca} \\
    \leq & 
  \left( 1 + \tfrac{k}{\epsilon} \right) \| \MM u \|_{\mcm} \| \MM v \|_{\mcm}
  + Fr^2 k \| \MM \eta \|_{\mca} \| \nabla \cdot \MM v \|_{\mca}
  + k \| \mcb \MM u \| \| \MM v \| \\
  & + Fr^2 \| \MM \eta \|_{\mca} \| \MM w \|_{\mca} 
  + Fr^2 k \| \nabla \cdot \MM u \|_{\mca} \| \| \MM w \|_{\mca}.
  \end{split} 
  \end{equation}
  At this point, we use the boundedness of $\mcb$ assumed in~\eqref{eq:bbd} and the norm equivalence of $\| \cdot \|$ and $\| \cdot \|_{\mcm}$ in~\eqref{eq:normequiv} to obtain
  \begin{equation}
      \begin{split}
          \hat{a}\left((\MM u , \MM \eta), (\MM v,  \MM w)\right)
          \leq &
          \left( 1 + \tfrac{k}{\epsilon} + \tfrac{k B^*}{C_\mcm^2} \right) \| \MM u \|_{\mcm} \| \MM v \|_{\mcm}
    + Fr^2 k \| \MM \eta \|_{\mca} \| \nabla \cdot \MM v \|_{\mca} \\
    & + Fr^2 \| \MM \eta \|_{\mca} \| \MM w \|_{\mca} 
    + Fr^2 k \| \nabla \cdot \MM u \|_{\mca} \| \| \MM w \|_{\mca}.
      \end{split}
  \end{equation}
  We can rewrite the right-hand side of this as the inner product of two vectors and apply discrete Cauchy-Schwarz to bound this by
  \begin{equation}
  \begin{split}
  \hat{a}\left((\MM u , \MM \eta), (\MM v,  \MM w)\right)
  & \leq  
  \sqrt{\left(1 + \tfrac{k}{\epsilon} + \frac{kB^*}{C_\mcm^2} \right) \| \MM u \|^2_{\mcm}
  + 2 Fr^2 \| \MM \eta \|^2_{ \mca} + Fr^2 k^2 \| \nabla \cdot \MM u \|^2_{\mca}}
   \\
  & \times \sqrt{\left(1 + \tfrac{k}{\epsilon} + \frac{kB^*}{C_\mcm^2} \right) \| \MM v \|^2_{\mcm}
  + 2 Fr^2 \| \MM w \|^2_{\mca} + Fr^2 k^2 \| \nabla \cdot \MM v \|^2_{\mca}}
    \\
  & \leq  C \left\| \left( \MM u, \MM \eta \right) \right\|_{\hat{b}}
  \left\| \left( \MM v, \MM w \right) \right\|_{\hat{b}}.
  \end{split}
  \end{equation}
\end{proof}
As a remark, it is possible to include the damping term in the inner product $\hat{b}$, in which case the continuity estimate is independent of $B^*$.  
However, typical use cases have small damping and the differences in resulting preconditioner performance are small.  
Furthermore, when damping is nonlinear, omitting it avoids the need to reassemble the preconditioner at each linear iteration.

\begin{theorem}
  \label{thm:lower}
The bilinear form $\hat{a}$ is inf-sup stable with respect to the $\| \cdot \|_{\hat{b}}$ norm with constant no smaller than $\tfrac{1}{2\sqrt{3}}$.
\end{theorem}
\begin{proof}
  Let $(\MM u, \MM \eta) \in \MM V_h^N \times \MM W_h^N$ be given and put
  $\MM v = \MM u$ and $\MM w = \MM \eta + k \nabla \cdot \MM u$.  Then we see that
  \begin{equation}
  \begin{split}
      \hat{a}\left((\MM u , \MM \eta), (\MM v,  \MM w)\right)
      = & \left( \MM u , \MM u \right)_\mcm + \tfrac{k}{\epsilon} \left( \MM u^\perp, \MM u \right)_\mcm - Fr^2 k \left( \MM \eta , \nabla \cdot \MM u \right)_{\mca}
      + k \left( \mcb \MM u , \MM u \right)  \\
      & + Fr^2\left( \MM \eta , \MM \eta + k \nabla \cdot \MM u \right)_{\mca}
      + Fr^2 k \left( \nabla \cdot \MM u , \MM \eta + k \nabla \cdot \MM u \right)_{\mca} \\
      = & \| \MM u \|_\mcm^2 + k \left( \mcb \MM u , \MM u \right)  \\
      & + Fr^2 \| \MM \eta \|_{\mca}^2 + Fr^2 k \left( \MM \eta , \nabla \cdot \MM u \right)_{\mca}
      + Fr^2 k^2 \| \nabla \cdot \MM u \|^2_{\mca}.
  \end{split}
  \end{equation}
  Now, the semi-definiteness of $\mcb$ and standard estimates let us make the bound
  \begin{equation}
  \begin{split}
      \hat{a}\left((\MM u , \MM \eta), (\MM v,  \MM w)\right)
       \geq & \| \MM u \|^2_{\mcm} 
      + Fr^2 \| \MM \eta \|^2_{\mca} + Fr^2 k^2 \| \nabla \cdot \MM u \|^2_{\mca}
      \\
      & - \tfrac{Fr^2}{2} \| \MM \eta \|^2_{\mca} - \tfrac{Fr^2 k^2}{2} \| \nabla \cdot \MM u \|^2_{\mca} \\
    = & \| \MM u \|^2_{\mcm} 
      + \tfrac{Fr^2}{2} \| \MM \eta \|^2_{\mca} + \tfrac{Fr^2 k^2}{2} \| \nabla \cdot \MM u \|^2_{\mca} \\
      \geq & \tfrac{1}{2} \left\| \left( \MM u , \MM \eta \right) \right\|_{\hat{b}}^2.
  \end{split}
  \end{equation}
  Now, we also have
  \begin{equation}
      \begin{split}
        \left\| \left( \MM v , \MM w \right) \right\|_{\hat{b}}^2
      & = \| \MM u \|^2_{\mcm} + Fr^2 k^2 \| \nabla \cdot \MM u \|^2_{\mca}
      + Fr^2 \| \MM \eta + k \nabla \cdot \MM u \|^2_{\mca} \\
      & \leq \| \MM u \|^2_{\mcm} + Fr^2 k^2 \| \nabla \cdot \MM u \|^2_{\mca}
      + 2 Fr^2 \left( \| \MM \eta \|_{\mca}^2 + k^2 \| \nabla \cdot \MM u \|^2_{\mca} \right)\\
      & \leq 3 \left\| \left( \MM u , \MM \eta \right) \right\|_{\hat{b}}^2.
      \end{split}
  \end{equation}
  Hence,
  \begin{equation}
      \hat{a}\left(\left( \MM u , \MM \eta \right), \left( \MM v , \MM w \right) \right)
      \geq \tfrac{1}{2} \left\| \left( \MM u , \MM \eta \right) \right\|
      \left\| \left( \MM v , \MM w \right) \right\|
      \geq \frac{1}{2\sqrt{3}} \left\| \left( \MM u , \MM \eta \right) \right\|_{\hat{b}}
      \left\| \left( \MM v , \MM w \right) \right\|_{\hat{b}},
  \end{equation}
  and the result follows.
\end{proof}

\section{More about $\mca$}
The major cost of applying our block diagonal preconditioner is the inversion of the upper-left block of~\eqref{eq:precond}:
\begin{equation}
\label{eq:C}
    \mathrm{C} = \mathrm{M}^V + Fr^2 k^2 \mathrm{E}^{\mca}.
\end{equation}
One could adapt the $\hdiv$ multigrid in~\cite{arnold2000multigrid} to this problem, but the requisite patch problems would include degrees of freedom of all the layers.  Solving the local patch problems would then become increasingly expensive as the number of layers increases.
We do not analyze this method further, but work toward approaches that avoid this limitation.
In passing, we also note a passing structural similarity of~\eqref{eq:C} to the matrices obtained for higher-order Runge-Kutta discretizations, so that it might be possible to adapt preconditioning techniques from references such as~\cite{farrell2020irksome, rana2021new, southworth1}.

In this section, we give an explicit formula for the inverse of $\mca$ and estimates on its extremal eigenvalues.
This sets us up to discuss preconditioners for $\mathrm{C}$ in the following section.

\subsection{An explicit inverse for \texorpdfstring{$\mca$}{A}}

\begin{proposition}
  Define the matrix $\mcc$ to be the $N \times N$ symmetric tridiagonal matrix with
  \begin{equation}
  \label{eq:cdiag}
    \mcc_{ii} = \begin{cases}
      \frac{1}{\rho_1} + \frac{1}{\rho_2 - \rho_1}, & i = 1, \\
      \frac{1}{\rho_i - \rho_{i-1}} + \frac{1}{\rho_{i+1} - \rho_i}, & 2 < i < N-1, \\
      \frac{1}{\rho_{n} - \rho_{n-1}}, & i = N,
    \end{cases}
  \end{equation}
  and off-diagonal entries
  \begin{equation}
  \label{eq:coffdiag}
    \mcc_{i,i+1} = \mcc_{i+1, i}  = -\frac{1}{\rho_{i+1} - \rho_i}, \ \ \ 1 \leq i \leq N-1.
  \end{equation}
  Then $\mcc$ is the inverse of $\mca$ given in~\eqref{eq:Adef}.
\end{proposition}
\begin{proof}
  The result can be obtained by Gauss-Jordan elimination on $\mca$, although the notation for the case of general $N$ is quite cumbersome.  Here, we confirm the result is correct by verifying $\mcc\mca=I$.

Since the diagonal of $\mcc$ is defined piecewise, we proceed in a few cases.  Consider the first row of $\mcs=\mcc\mca$:
\begin{equation}
\begin{split}
  \mcs_{11} & = \mcc_{11} \mca_{11} + \mcc_{12} \mca_{21} \\
  & = \left( \frac{1}{\rho_1} + \frac{1}{\rho_2 - \rho_1} \right) \rho_1
  - \frac{1}{\rho_2 - \rho_1} \rho_1 = 1. 
\end{split}
\end{equation}
For any $j > 1$, we have that $\mca_{1j} = \rho_1$ and $\mca_{2j} = \rho_2$, so
\begin{equation}
\begin{split}
  \mcs_{1j} & = \mcc_{11} \mca_{1j} + \mcc_{12} \mca_{2j} \\
  & = \left( \frac{1}{\rho_1} + \frac{1}{\rho_2 - \rho_1} \right) \rho_1
  - \frac{1}{\rho_2 - \rho_1} \rho_2 \\
  & = 1 - \frac{\rho_2 - \rho_1}{\rho_2 - \rho_1} = 0.
\end{split}
\end{equation}
Now, for $2 \leq i < N$,  we have
\begin{equation}
  \begin{split}
    \mcs_{ii} & = \mcc_{i, i-1} \mca_{i-1, i} + \mcc_{i, i} \mca_{i, i} + \mcc_{i, i+1} \mca_{i+1, i} \\
    & = -\frac{1}{\rho_{i} - \rho_{i-1}} \rho_{i-1}
    + \left(\frac{1}{\rho_{i} - \rho_{i-1}} + \frac{1}{\rho_{i+1} - \rho_{i}} \right)
    \rho_i
    -\frac{1}{\rho_{i+1} - \rho_{i}} \rho_i = 1,
  \end{split}
\end{equation}
For some $j > i$, we have
\begin{equation}
  \begin{split}
    \mcs_{ij} & = \mcc_{i, i-1} \mca_{i-1, j} + \mcc_{i, i} \mca_{i, j} + \mcc_{i, i+1} \mca_{i+1, j} \\
    & = -\frac{1}{\rho_{i} - \rho_{i-1}} \rho_{i-1}
    + \left(\frac{1}{\rho_{i} - \rho_{i-1}} + \frac{1}{\rho_{i+1} - \rho_{i}} \right)
    \rho_i
    -\frac{1}{\rho_{i+1} - \rho_{i}} \rho_{i+1} = 0,
  \end{split}
\end{equation}
and for $j < i$,
\begin{equation}
  \begin{split}
    \mcs_{ij} & = \mcc_{i, i-1} \mca_{i-1, j} + \mcc_{i, i} \mca_{i, j} + \mcc_{i, i+1} \mca_{i+1, j} \\
    & = -\frac{1}{\rho_{i} - \rho_{i-1}} \rho_{j}
    + \left(\frac{1}{\rho_{i} - \rho_{i-1}} + \frac{1}{\rho_{i+1} - \rho_{i}} \right)
    \rho_j
    -\frac{1}{\rho_{i+1} - \rho_{i}} \rho_{j} = 0.
  \end{split}
\end{equation}
Finally, we handle the last row.  The diagonal entry there is
\begin{equation}
  \begin{split}
    \mcs_{N,N} & = \mcc_{N, N-1} \mca_{N-1, N} + \mcc_{N, N} \mca_{N, N} \\
    & = -\frac{1}{\rho_N - \rho_{N-1}} \rho_{N-1} + \frac{1}{\rho_N - \rho_{N-1}} \rho_{N} = 1,
  \end{split}
\end{equation}
and for any $1 \leq j < N$,
\begin{equation}
  \begin{split}
    \mcs_{N,N} & = \mcc_{N, N-1} \mca_{N-1, j} + \mcc_{N, N} \mca_{N, j} \\
    & = -\frac{1}{\rho_N - \rho_{N-1}} \rho_{j} + \frac{1}{\rho_N - \rho_{N-1}} \rho_{j} = 0.
  \end{split}
\end{equation}
\end{proof}

Finally, we note that since $\mcc$ is tridiagonal and symmetric positive-definite, it has a factorization 
\begin{equation}
\label{eq:ldlt}
    \mcc = \mcl \mcd \mcl^T
\end{equation} 
with bidiagonal $\mcl$ and diagonal $\mcd$ with positive entries.

\subsection{The spectrum of \texorpdfstring{$\mca$}{A}}
Subsequent analysis will rely on knowing things about the spectrum of $\mca$, and we are able to give certain instructive spectral bounds here.  Since $\mca$ is symmetric and positive-definite, we let $\lambda_1 \geq \lambda_2 \geq \dots \geq \lambda_N > 0$ be its eigenvalues, arranged in nonincreasing order.

\begin{proposition}
The largest eigenvalue of $\mca$ satisfies
\begin{equation}
\label{eq:lam1bound}
    N \rho_1 \leq \lambda_1 \leq \sum_{j=1}^N \rho_j
\end{equation}
\end{proposition}
\begin{proof}
We handle the upper bound by Gerschgorin's Circle Theorem.  Owing to the structure of $\mca$, the largest outer extent of a Gerschgorin disk comes from the final row, and the maximal value is
\begin{equation}
\rho_N + \sum_{j=1}^{N-1} \rho_j = \sum_{i=1}^N \rho_i.
\end{equation}

Now, we derive the lower bound in~\eqref{eq:lam1bound}, which confirms that $\lambda_1$ is in fact comparable to $N$.  Since $\lambda_1$ maximizes the Rayleigh quotient:
\begin{equation}
    \lambda_1 = \max_{\bfx \neq 0} \frac{\bfx^T A \bfx}{\bfx^T \bfx},
\end{equation}
using any particular choice of nonzer $\bfx$ in the Rayleigh quotient gives a lower bound for $\lambda_1$.
We chose the vector $\bfx$ consisting entire of ones.
Since $\bfx^T \bfx = N$, we know that
\begin{equation}
  N \lambda_1 \geq \bfx^T \mca \bfx.
\end{equation}
Proceeding, the entries of $\mca \bfx$ are just the row sums of $\mca$:
\begin{equation}
\left( \mca \bfx \right)_i = \sum_{j=1}^{i-1} \rho_j + \sum_{j=i}^N \rho_i
= \left( \sum_{j=1}^{i-1} \rho_j \right) + \left(N-i+1\right) \rho_i.
\end{equation}
Evaluating $\bfx^T \mca \bfx$ gives
\begin{equation}
\begin{split}
    \bfx^T \mca \bfx & = \sum_{i=1}^N \left( \mca \bfx \right)_i
    = \sum_{i=1}^N \left[ \left( \sum_{j=1}^{i-1} \rho_j \right) + \left(N-i+1\right) \rho_i \right] \\
    & = \sum_{i=1}^N \left(N - i\right) \rho_i + \sum_{i=1}^N \left(N-i+1\right)\rho_i
    = \sum_{i=1}^N \left( 2N - 2i + 1 \right) \rho_i.
    \end{split} 
\end{equation}
Since $\rho_1 < \rho_i$ for $i > 1$, 
\begin{equation}
    \begin{split}
        N \lambda_1 & \geq \rho_1 \sum_{i=1}^N \left( 2 N - 2i + 1 \right)
        = \rho_1 \left[ 2 N^2 - 2 \frac{N(N+1)}{2} + N \right]
         = N^2 \rho_1.
    \end{split} 
\end{equation}
This proves the lower bound.
\end{proof}
Similar techniques can lead to upper and lower bounds on the minimal eigenvalue $\lambda_N$:
\begin{theorem}
Let $\delta \rho_* = \min_{1\leq i \leq N-1} \rho_{i+1}-\rho_i$ and $\delta \rho^* = \max_{1\leq i \leq N_1} \rho_{i+1}-\rho_i$.  Then
\begin{equation}
\label{eq:lamnbound}
\tfrac{\delta \rho_*}{4} \leq \lambda_N \leq \tfrac{3 \delta \rho^*}{10}.
\end{equation}
\end{theorem}
\begin{proof}
We apply the Gerschgorin Circle Theorem to bound the maximal eigenvalue of $\mcc$, which  is the reciprocal of the minimal eigenvalue of $\mca$, to give the claimed lower bound.
Consider the first row of $\mcc$.  The diagonal plus sum of magnitudes of off-diagonal entries yields
\begin{equation}
\tfrac{1}{\rho_1} + \tfrac{1}{\rho_2-\rho_1} = 
\tfrac{\rho_1+\rho_2}{\rho_1(\rho_2-\rho_1)}.
\end{equation}
We then use~\eqref{eq:rhonnotbig} to bound this by 
\[
\tfrac{3}{\rho_2-\rho_1} \leq \tfrac{3}{\delta \rho_*}.
\]
Then, for $2 \leq i \leq N-1$, the diagonal plus sum of off-diagonal magnitudes gives
\begin{equation}
2 \left[ \tfrac{1}{\rho_i - \rho_i-1} + \tfrac{1}{\rho_{i+1}-\rho_i} \right]
\leq \tfrac{4}{\delta \rho_*}.
\end{equation}
Finally, outer limit of the Gerschgorin disk for the final row is
\begin{equation}
    \tfrac{2}{\rho_N - \rho_{N-1}} \leq \tfrac{2}{\delta \rho_*}.
\end{equation}
Taking the maximum over these three calculations gives that 
\begin{equation}
    \tfrac{1}{\lambda_N} \leq \tfrac{4}{\delta \rho_*},
\end{equation}
and the reciprocal of this inequality gives the lower bound.

To establish the upper bound, we again consider the Rayleigh quotient on a particular vector.  Pick some vector $\mathbf{x}$ such that for a fixed $3 \leq i \leq N-2$
\begin{equation}
    \bfx_j = \begin{cases}
    1, & j = i, \\
    -1, & |j-i| = 1, \\
    0, & \mathrm{otherwise}.
    \end{cases}
\end{equation}
Selecting  $i =1,2,N-2,N-1$, although this requires dealing with exceptional first and last rows of~\eqref{eq:cdiag} and does not appreciably improve our bound.  Since $\bfx$ is nonzero only in entries $i-1, i, i+1$, we directly computing the relevant entries of $\mathcal{C}\bfx$ using~\eqref{eq:cdiag} and~\eqref{eq:coffdiag}.
\begin{equation}
\begin{split}
    \left( \mcc \bfx \right)_{i-1}
    & = \sum_{j=1}^N \mcc_{i-1,j} \bfx_j
    = \mcc_{i-1, i-1} \bfx_{i-1} + \mcc_{i-1, i} \bfx_i \\
    & = -\left( \tfrac{1}{\rho_{i-1}-\rho_{i-2}} + \tfrac{1}{\rho_{i}-\rho_{i-1}} \right)
    - \tfrac{1}{\rho_i - \rho_{i-1}} \\
    & = -\tfrac{1}{\rho_{i-1} - \rho_{i-2}} - \tfrac{2}{\rho_i - \rho_{i-1}}.
\end{split}
\end{equation}
\begin{equation}
    \begin{split}
        \left(\mcc \bfx\right)_i & = \sum_{j=1}^N \mcc_{i,j} \bfx_j
    = \mcc_{i, i-1} \bfx_{i-1} + \mcc_{i, i} \bfx_i + \mcc_{i, i+1} \bfx_{i+1} \\
    & = \tfrac{1}{\rho_{i}-\rho_{i-1}} + \left( \tfrac{1}{\rho_i-\rho_{i-1}} + \tfrac{1}{\rho_{i+1}-\rho_i} \right) + \tfrac{1}{\rho_{i+1}-\rho_i} \\
    & = \tfrac{2}{\rho_{i}-\rho_{i-1}} + \tfrac{2}{\rho_{i+1}-\rho_i}.
    \end{split}
\end{equation}
Similarly, we can compute
\begin{equation}
    \left( \mcc \bfx \right)_{i+1} = -\tfrac{2}{\rho_{i+1}-\rho_i} - \tfrac{1}{\rho_{i+2}-\rho_{i+1}}.
\end{equation}
Now, we use the results to directly calculate that
\begin{equation}
    \bfx^T \mcc \bfx = \tfrac{1}{\rho_{i-1}-\rho_{i-2}} + \tfrac{4}{\rho_{i}-\rho_{i-1}}
    + \tfrac{4}{\rho_{i+1}-\rho_{i}} + \tfrac{1}{\rho_{i+2}-\rho_{i+1}}
    \geq \tfrac{10}{\delta \rho^*}.
\end{equation}
Now, we note that $\bfx^T \bfx = 3$ for this choice of $\bfx$ and using the Rayleigh quotient gives the upper bound on $\lambda_N$.
\end{proof}
Assuming some kind of comparability between $\delta \rho_*$ and $\delta \rho_*$, both are on the order of $N$.  This gives a spectral condition number (ratio of extremal eigenvalues) for $\mca$ on the order of $N^{2}$.

\section{Simplifying the preconditioner}
Our weighted norm preconditioner~\eqref{eq:precond} provides parameter-robustness, but also maintains an all-to-all coupling between the layers that can become expensive as the number of layers increases.
In this section, we propose two approaches to overcoming this difficulty.
In the first case, we simply ignore the inter-layer coupling.
We are able to prove that this strategy is more effective than the the $N^2$ conditioning of $A$ might otherwise suggest.
In the second case, we make use of the special properties of $\mca$ derived above to propose a change of variables in the upper-left block of~\eqref{eq:precond} that renders coupling only between adjacent layers.

\subsection{Neglecting inter-layer coupling}
The bilinear form
\begin{equation}
\label{eq:blc}
    c(\MM u, \MM v) = \left( \MM u , \MM v \right)_\mcm + Fr^2 k^2 \left( \nabla \cdot \MM u , \nabla \cdot \MM v \right)_{\mca},
\end{equation}
yields the matrix~\eqref{eq:C} under discretization, and we want to
compare $c$ to the simpler form obtained by replacing the $\mca$-weighted inner product with the standard one:
\begin{equation}
  \label{eq:noa}
\hat{c}(\MM u, \MM v) = \left( \MM u , \MM v \right)_\mcm + Fr^2 k^2 \left( \nabla \cdot \MM u , \nabla \cdot \MM v \right).
\end{equation}
The latter form gives rise to the block diagonal matrix
\begin{equation}
\hat{\mathrm{C}} = \mathrm{M}^V + Fr^2 k^2  \mathrm{E},
\end{equation}
which we can consider using it as a preconditioner for the matrix derived from $c(\cdot, \cdot)$.
Both $c$ and $\hat{c}$ are symmetric and positive-definite, and showing an equivalence between them controls eigenvalues of the system obtained by preconditioning one with the other.

As a first attempt,
$\mca$ is symmetric and positive-definite, and we can use the Rayleigh quotient pointwise inside of integrals to obtain:
\begin{equation}
  \| \MM w \|^2_{\mca}
  = \int_\Omega \left( \mca \MM w \right) \cdot \MM w \, dx
  \geq \int_\Omega \lambda_1 \left( \MM w \cdot \MM w \right)
  = \lambda_1 \| \MM w \|^2,
\end{equation}
with a similar upper bound of $\| \MM w \|^2_{\mca} \leq \lambda_1 \| \MM w \|^2.$

Using this observation, 
\begin{equation}
  \label{eq:boobound}
\lambda_N \left( \| \MM u \|^2_{\mcm} + k^2 Fr^2 \| \nabla \cdot \MM u \|^2 \right)
\leq 
c(\MM u, \MM u)
\leq \lambda_1 \left( \| \MM u \|^2_{\mcm} + k^2 Fr^2 \| \nabla \cdot \MM u \|^2 \right),
\end{equation}
so that an equivalence between $c$ and $\hat{c}$ holds with a condition number of $\lambda_1 / \lambda_N$, which is quadratic in the number of layers.
With more careful consideration, however, are able to prove a tighter bound.

In this analysis, we will make the \emph{inverse assumption} that there exists some $C_I > 0$, independent of $\MM u$ and $h$ such that
\begin{equation}
    \label{eq:inverse}
    \| \nabla \cdot \MM u \| \leq \tfrac{C_I}{h} \| \MM u \|_\mcm
\end{equation}
holds for all $\MM u \in \MM V^N_h$ with some $C_I > 0$ independent of $\MM u$.
This estimate is a theorem for standard $H^1$ polynomial spaces~\cite{Brenner:2008} and is commonly made assumption for $\hdiv$ spaces.  In our case, it follows from the standard $\hdiv$ inverse assumption in each component plus the equivalence of $\| \cdot \|_\mcm$ to the $N$-way $L^2$ inner product.

To simplify our notation, we introduce the quantity
\begin{equation}
  q = C_I k Fr.
\end{equation}

\begin{theorem}
  \label{thm:noa}
  For all $\mathbf{u} \in \MM V^N_h$, the equivalence
  \begin{equation}
    \chi_0 \hat{c}(\mathbf{u}, \mathbf{u})
    \leq c(\mathbf{u}, \mathbf{u}) \leq \chi_1 \hat{c}(\mathbf{u}, \mathbf{u})
  \end{equation}
  holds, where
  \begin{equation}
      \chi_0 = \frac{\lambda_N q^2 + h^2}{q^2 + h^2}, \ \ \
      \chi_1 = \frac{\lambda_1 q^2 + h^2}{q^2 + h^2}.
  \end{equation}
\end{theorem}

\begin{proof}
  We first prove the upper bound involving $\chi_1$,
  applying the Rayleigh quotient for $\mathcal{A}$ pointwise to obtain
  \begin{equation}
    c(\MM u, \MM u) \leq \| \MM u \|_\mcm^2
    + \lambda_1 Fr^2 k^2 \| \nabla \cdot \MM u \|^2.
  \end{equation}
  Next, for some $0 \leq \alpha \leq 1$ to be specified, we split the $\| \nabla \cdot \MM u \|^2$ term
  \begin{equation}
    c(\MM u, \MM v) 
    \leq \| \MM u \|_\mcm^2
    + \alpha \lambda_1 Fr^2 k^2 \| \nabla \cdot \MM u \|^2
    + (1-\alpha) \lambda_1 Fr^2 k^2 \| \nabla \cdot \MM u \|^2,
  \end{equation}
  and using the inverse assumption~\eqref{eq:inverse}, we have
  \begin{equation}
    c(\MM u, \MM v) 
    \leq \left( 1 + \tfrac{\alpha \lambda_1 q^2}{h^2} \right)\| \MM u \|_\mcm^2
    + \left(1-\alpha\right) \lambda_1 Fr^2 k^2 \| \nabla \cdot \MM u \|^2.
  \end{equation}
  The best bound here will be obtained if we choose $\alpha$ to equalize the coefficients of the terms appearing in the bilinear form, or
  \[
   1 + \tfrac{\alpha \lambda_1 q^2}{h^2} = \left( 1 - \alpha\right) \lambda_1.
   \]
   This is readily solved to find
   \begin{equation}
     \alpha = \frac{\left( \lambda_1 - 1 \right)h^2}{\lambda_1\left( h^2 + q^2 \right)}.
   \end{equation}
   So then, the coefficient of $k^2 Fr^2 \| \nabla \cdot \MM u \|^2$ in our estimate is $\alpha_1 \lambda_1$, which is equal to the claimed value $\chi_1$.
   The coefficient of $\| \MM u \|_\mcm^2$ must have the same value, completing the upper bound.

   Now, we consider the lower bound, which begins in the same way, using the lower bound on the Rayleigh quotient to write
   \begin{equation}
    c(\MM u, \MM v) \geq \| \MM u \|_\mcm^2
    + \lambda_N Fr^2 k^2 \| \nabla \cdot \MM u \|^2.
   \end{equation}
   Now, we additively split the $L^2$ term with some $0 < \alpha < 1$:
  \begin{equation}
    c(\MM u, \MM u) \geq (1-\alpha) \| \MM u \|_{\mcm}^2
    + \alpha \| \MM u \|_{\mcm}^2
    + \lambda_N k^2 Fr^2 \| \nabla \cdot \MM u \|^2.
  \end{equation}
  Now, we rearrange the inverse assumption to bound $\| \MM u \|_\mcm$ below by $\tfrac{h}{C_I} \| \nabla \cdot \MM u \|$
  \begin{equation}
    \begin{split}
    c(\MM u, \MM u) & \geq (1-\alpha) \| \MM u \|_{\mcm}^2
    + \left( \tfrac{\alpha h^2}{C_I^2} + k^2 Fr^2 \lambda_N \right) \left\| \nabla \cdot \MM u \right\|^2 \\
    & = (1-\alpha) \| \MM u \|_{\mcm}^2
    + \left( \tfrac{\alpha h^2}{q^2} + \lambda_N \right) k^2 Fr^2 \left\| \nabla \cdot \MM u \right\|^2,
    \end{split}
  \end{equation}
  Again, the optimal choice of $\alpha$ will balance the coefficients, so we solve
  \[
  1 - \alpha = \tfrac{\alpha h^2}{q^2} + \lambda_N
  \]
  to find
  \begin{equation}
    \alpha = \frac{q^2 \left( 1 - \lambda_N \right)}{q^2 + h^2},
  \end{equation}
  so that $1 - \alpha = \chi_0$ as claimed.
\end{proof}

This theorem shows a somewhat complex relationship between the physical and discretization parameters and the equivalence bounds obtained by neglecting the inter-layer coupling.
The lower bound is somewhat simpler to unpack.  Since $\lambda_N > 0$ but decays like $1/N$, we always have
\[
\chi_0 \geq \frac{h^2}{q^2 + h^2},
\]
which is \emph{independent} of the number of layers.
Fixing $h$ and letting $q$ (here, a proxy for the time step) become small presents no problems.
On the other hand, keeping a nondegenerate lower bound when $h \rightarrow 0$ also requires $q \rightarrow 0$ at a comparable rate.

The asymptotics of the upper bound are a bit different.
We have that $\lambda_1 = \mathcal{O}(N)$ as we increase the number of layers.  However, this only makes $\chi_1/\chi_0 = \mathcal{O}(N)$ rather than the naive $\mathcal{O}(N)^2$ posited initially.
Also, for a fixed number of layers, two comments are in order.
First, we always have
$\chi_1 < \lambda_1$.
Second, we can decrease the effect of large $\lambda_1$ by reducing the time step relative to the mesh size, for
\[
\chi_1 = \frac{\lambda_1 q^2 + h^2}{q^2 + h^2}
= \frac{\lambda_1 \left(\frac{q}{h}\right)^2 + 1}{\left(\frac{q}{h}\right)^2 + 1}.
\]

\subsection{A block tridiagonal reformulation}
Neglecting the inter-layer coupling in our preconditioner is better than initially thought, and performs well for practical numbers of layers.
Here, we sketch an alternate approach that should also sparsify the preconditioner while maintaining the layer-independence.
This approach relies heavily on the tridiagonal inverse of the coupling matrix $\mca$.

For the bilinear form $c$ from~\eqref{eq:blc} and some bounded linear functional $f \in (\MM V_h^N)^\prime$, consider the variational problem
\begin{equation}
\label{pcstart}
    c(\MM u, \MM v) = f(\MM v), \ \ \ \MM v \in V_h^N
\end{equation}
Using~\eqref{eq:ldlt} in this, we write
\begin{equation}
\mca = \mcc^{-1} = 
\left( \mcl \mcd \mcl^T \right)^{-1}
= \mcl^{-T} \mcd^{-1} \mcl^{-1},
\end{equation}
so that
\begin{equation}
\begin{split}
    c(\MM u, \MM v)
    & = (\mcm \MM u, \MM v) + Fr^2 k^2 ( \mcl^{-T} \mcd^{-1} \mcl^{-1} \nabla \cdot \MM u , \nabla \cdot \MM v) \\
    & = (\mcm \MM u, \MM v) + Fr^2 k^2 (  \mcd^{-1} \nabla \cdot \mcl^{-1} \MM u , \nabla \cdot \mcl^{-1} \MM v).
\end{split}
\end{equation}
Now, we introduce auxilary variables $\widetilde{\MM u} = \mcl^{-1} \MM u$ and $\widetilde{\MM v} = \mcl^{-1} \MM v$ and define the matrix $ \widetilde{\mcc} = \mcl^T \mcm \mcl$.  This varies spatially, but is pointwise tridiagonal.  With these substitutions, we have
\begin{equation}
  \begin{split}
    c(\MM u, \MM v) 
    & = (\mcm \mcl \widetilde{\MM u}, \mcl \widetilde{\MM v}) + Fr^2 k^2 (  \mcd^{-1} \nabla \cdot \widetilde{ \MM u} , 
    \nabla \cdot \widetilde{\MM v})  \\
    & = (\widetilde{\mcc} \widetilde{\MM u},  \widetilde{\MM v}) + Fr^2 k^2 (  \mcd ^{-1} \nabla \cdot \widetilde{ \MM u} , 
    \nabla \cdot \widetilde{\MM v}) \\
    & \equiv c(\widetilde{\MM u}, \widetilde{\MM v}) 
\end{split}
\end{equation}
Now, we can write~\eqref{pcstart} as 
\begin{equation}
\label{eq:pc2}
    \tilde{c}\left(\widetilde{\MM u}, \widetilde{\MM v}\right) = \tilde{f}(\widetilde{\MM v}).
\end{equation}
Hence, one could change variables and solve a sparser system, in which only adjacent layers are coupled through the tridiagonal matrix $\widetilde{\mcc}$, although this requires considerable care in the implementation.

\section{Numerical results}
We have implemented a mixed finite element discretization of the tide
model and developed all of our preconditioners within the Firedrake
framework~\cite{Rathgeber:2016}.  Firedrake is an automated system for
the solution of PDE using the finite element method.  It generates
efficient low-level code from the
Unifed Form Language (UFL) in Python~\cite{alnaes2014unified},
and interfaces tightly with PETSc for scalable algebraic solvers.
Firedrake also has a rich ability to interoperate with and extend
PETSc~\cite{kirby2018solver}, which facilites the definition of
auxiliary bilinear forms needed for weighted norm preconditioning.
Morever, a facility to generate Runge-Kutta methods from a
semi-discrete formulation was recently added to Firedrake through the
Irksome project~\cite{farrell2020irksome}, and we use this to
obtain the implicit midpoint rule.

Our numerical experiments primarily consist of testing
preconditioners as a function of discretization and physical parameters.
We discretize the problem on the unit square by taking an $N_x \times N_y$ mesh subdivided into
right triangles and use lowest-order Raviart-Thomas elements for $\MM
u_h$ and piecewise constants for $\MM \eta_h$.
In all our cases, we solve the resulting linear systems using
unrestarted GMRES with right preconditioning.
We chose the right-hand side by choosing an initial condition for the IBVP at rest but for a small disturbance in the top layer and taking one step of the implicit midpoint rule using Irksome.
We iterated to the PETSc default
relative tolerance of $10^{-5}$, which is appropriate for the
low-order time and space discretizations under consideration.  In
certain cases, we found it necessary to use modified Gram-Schmidt
orthogonalization, and so we used it throughout.
Our techniques are not particular to the Raviart-Thomas elements or triangles.  Much as in~\cite{kirby2021preconditioning}, we have also performed our experiments on rectangular Raviart-Thomas elements and trimmed serendipity elements~\cite{crum2022bringing, gillette2019computational} with very similar results.  

As point of reference, we will compare the weighted-norm
preconditioners under consideration to a standard incomplete LU
factorization method~\cite{saad2003iterative} with no fill.
(Firedrake natively stores the momentum and elevation variables separately,
but PETSc performs nested dissection to reorder the unknowns before performing the factorizations.)
For wave-like equations with a reasonable time step and moderate physical parameters,
this is not a terrible approach.
We refer to the two plots in Figure~\ref{fig:ilu}.
Both plots fix 5 layers with equidistributed densities between 1.03 and 1.06 and Rossby number $\epsilon = 1$.
In the first plot, we vary the Froude number and in the second, we vary the CFL number $\Delta t/N$.
In both cases, we have mesh independence, but we see a wide range of variation with respect to the physical and discretization parameters.

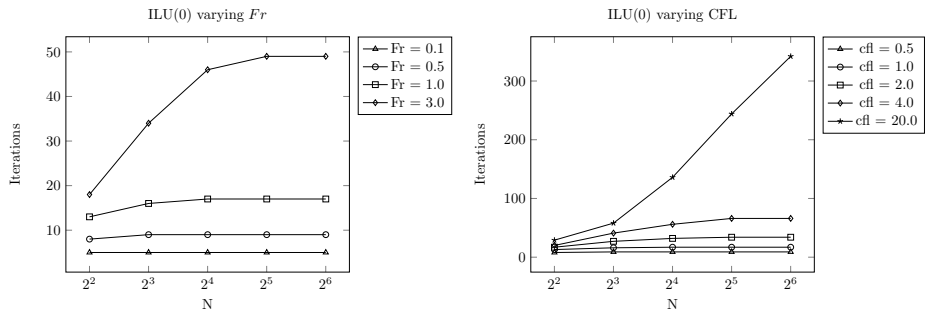
\begin{figure}
  \begin{subfigure}[c]{0.475\textwidth}
    \begin{tikzpicture}[scale=0.55]
      \begin{axis}[title = {ILU(0) varying $Fr$},
          xmode = log,
          log basis x={2},
          xlabel = {N}, ylabel = {Iterations},
          legend pos = outer north east]
        \addlegendentry{Fr = 0.1}
        \addplot[mark = triangle] table [x = {N}, y = 0.1, col sep = comma] {rt_ilu_fr.csv};
        
        \addlegendentry{Fr = 0.5}
        \addplot[mark = o] table [x = {N}, y = 0.5, col sep = comma] {rt_ilu_fr.csv};
        
        \addlegendentry{Fr = 1.0}
        \addplot[mark = square] table [x = {N}, y = 1.0, col sep = comma] {rt_ilu_fr.csv};
        
        \addlegendentry{Fr = 3.0}
        \addplot[mark = diamond] table [x = {N}, y = 3.0, col sep = comma] {rt_ilu_fr.csv};
        \end{axis}
        \end{tikzpicture}
  \end{subfigure}
  \begin{subfigure}[c]{0.475\textwidth}
	\begin{tikzpicture}[scale=0.55]
	\begin{axis}[title = {ILU(0) varying CFL},
	xmode = log,
	log basis x={2},
	xlabel = {N}, ylabel = {Iterations},
	legend pos = outer north east]
	\addlegendentry{cfl = 0.5}
	\addplot[mark = triangle] table [x = {N}, y = 0.5, col sep = comma] {rt_ilu_cfl.csv};
	
	\addlegendentry{cfl = 1.0}
	\addplot[mark = o] table [x = {N}, y = 1.0, col sep = comma] {rt_ilu_cfl.csv};
	
	\addlegendentry{cfl = 2.0}
	\addplot[mark = square] table [x = {N}, y = 2.0, col sep = comma] {rt_ilu_cfl.csv};
	
	\addlegendentry{cfl = 4.0}
	\addplot[mark = diamond] table [x = {N}, y = 4.0, col sep = comma] {rt_ilu_cfl.csv};
	
	\addlegendentry{cfl = 20.0}
	\addplot[mark = star] table [x = {N}, y = 20.0, col sep = comma] {rt_ilu_cfl.csv};
	
	\end{axis}
	\end{tikzpicture}
  \end{subfigure}
  \caption{Performance of ILU(0) preconditioner as a function of mesh parameters for various Froude and CFL numbers.  Throught, we fix $\epsilon = 1$ and consider 5 layers with densities varying between 1.03 and 1.06.
    We see eventual mesh independence, but the number of iterations varies considerably with fixing the CFL number as 1 and varying the Froude number (left) or fixing $Fr=1$ and varying the CFL number (right).}
  \label{fig:ilu}
\end{figure}

We repeat these same experiments, now with the weighted-norm
preconditioner we proposed in~\eqref{eq:precond}.
Applying this preconditioner requires at least approximately inverting the block diagonal matrix.
The best (in terms of iteration count) we can hope for is obtained if those blocks are in fact inverted exactly.  The bottom right block is diagonal for lowest-order elements and hence trivial to invert.
For the top left block, we compute a sparse LU factorization in a setup phase and perform solves with this at each iteration.
We can compare Figure~\ref{fig:wtdnrm} to those in~\ref{fig:ilu} and see the potential benefit of our new preconditioner.
Although we see some variation with respect to the Froude and CFL numbers, we seem to approach a relatively small and mesh-independent bound, even for rather extreme parameter values.

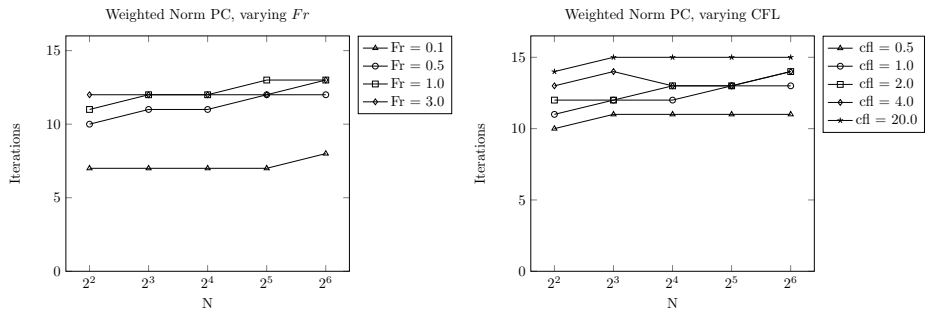
\begin{figure}
  \begin{subfigure}[c]{0.475\textwidth}
	\begin{tikzpicture}[scale=0.55]
	\begin{axis}[title = {Weighted Norm PC, varying $Fr$},
	xmode = log,
	log basis x={2}, ymin=0, ymax=16,
	xlabel = {N}, ylabel = {Iterations},
	legend pos = outer north east]
	\addlegendentry{Fr = 0.1}
	\addplot[mark = triangle] table [x = {N}, y = 0.1, col sep = comma] {rt_wtd_nrm_lu_fr.csv};
	
	\addlegendentry{Fr = 0.5}
	\addplot[mark = o] table [x = {N}, y = 0.5, col sep = comma] {rt_wtd_nrm_lu_fr.csv};
	
	\addlegendentry{Fr = 1.0}
	\addplot[mark = square] table [x = {N}, y = 1.0, col sep = comma] {rt_wtd_nrm_lu_fr.csv};
	
	\addlegendentry{Fr = 3.0}
	\addplot[mark = diamond] table [x = {N}, y = 3.0, col sep = comma] {rt_wtd_nrm_lu_fr.csv};
	
	\end{axis}
	\end{tikzpicture}
  \end{subfigure}
  \begin{subfigure}[c]{0.475\textwidth}
	\begin{tikzpicture}[scale=0.55]
	\begin{axis}[title = {Weighted Norm PC, varying CFL},
	xmode = log,
	log basis x={2}, ymin = 0,
	xlabel = {N}, ylabel = {Iterations},
	legend pos = outer north east]
	\addlegendentry{cfl = 0.5}
	\addplot[mark = triangle] table [x = {N}, y = 0.5, col sep = comma] {rt_wtd_nrm_lu_cfl.csv};
	
	\addlegendentry{cfl = 1.0}
	\addplot[mark = o] table [x = {N}, y = 1.0, col sep = comma] {rt_wtd_nrm_lu_cfl.csv};
	
	\addlegendentry{cfl = 2.0}
	\addplot[mark = square] table [x = {N}, y = 2.0, col sep = comma] {rt_wtd_nrm_lu_cfl.csv};
	
	\addlegendentry{cfl = 4.0}
	\addplot[mark = diamond] table [x = {N}, y = 4.0, col sep = comma] {rt_wtd_nrm_lu_cfl.csv};
	
	\addlegendentry{cfl = 20.0}
	\addplot[mark = star] table [x = {N}, y = 20.0, col sep = comma] {rt_wtd_nrm_lu_cfl.csv};
	
	\end{axis}
	\end{tikzpicture}    
  \end{subfigure}
  \caption{Performance of the preconditioner~\eqref{eq:precond} using exact inversion of the blocks.  Parameters are the same as in Figure~\ref{fig:ilu}.}
  \label{fig:wtdnrm}
\end{figure}

However, for scaling to very large problems, it is important to consider ways of bypassing sparse factorizations.
A simple strategy for this is to replace the inversion of the top left block with a simple ILU(0) factorization, and we repeat the experiments from Figures~\ref{fig:ilu} and~\ref{fig:wtdnrm} using this choice in Figure~\ref{fig:wtdnrmilu}.
As expected, we lose some parameter robustness, but this could still give a practical result.
These plots show that even running at CFL number requires only about 20 iterations per time step, and ILU(0) costs about as much as a matrix-vector product to apply.
We remark that some adaptation of $\hdiv$ multigrid~\cite{Arnold1997} could recover parameter robustness at the cost of more expensive iterations.

\begin{figure}
  \begin{subfigure}[c]{0.475\textwidth}
	\begin{tikzpicture}[scale=0.55]
	\begin{axis}[title = {Weighted Norm/ILU varying $Fr$},
	xmode = log,
	log basis x={2},
	xlabel = {N}, ylabel = {Iterations},
	legend pos = outer north east]
	\addlegendentry{Fr = 0.1}
	\addplot[mark = triangle] table [x = {N}, y = 0.1, col sep = comma] {rt_wtd_nrm_ilu_fr.csv};
	
	\addlegendentry{Fr = 0.5}
	\addplot[mark = o] table [x = {N}, y = 0.5, col sep = comma] {rt_wtd_nrm_ilu_fr.csv};
	
	\addlegendentry{Fr = 1.0}
	\addplot[mark = square] table [x = {N}, y = 1.0, col sep = comma] {rt_wtd_nrm_ilu_fr.csv};
	
	\addlegendentry{Fr = 3.0}
	\addplot[mark = diamond] table [x = {N}, y = 3.0, col sep = comma] {rt_wtd_nrm_ilu_fr.csv};
	
	\end{axis}
	\end{tikzpicture}    
  \end{subfigure}
  \begin{subfigure}[c]{0.475\textwidth}
	\begin{tikzpicture}[scale=0.55]
	\begin{axis}[title = {Weighted Norm/ILU varying CFL},
	xmode = log,
	log basis x={2},
	xlabel = {N}, ylabel = {Iterations},
	legend pos = outer north east]
	\addlegendentry{cfl = 0.5}
	\addplot[mark = triangle] table [x = {N}, y = 0.5, col sep = comma] {rt_wtd_nrm_ilu_cfl.csv};
	
	\addlegendentry{cfl = 1.0}
	\addplot[mark = o] table [x = {N}, y = 1.0, col sep = comma] {rt_wtd_nrm_ilu_cfl.csv};
	
	\addlegendentry{cfl = 2.0}
	\addplot[mark = square] table [x = {N}, y = 2.0, col sep = comma] {rt_wtd_nrm_ilu_cfl.csv};
	
	\addlegendentry{cfl = 4.0}
	\addplot[mark = diamond] table [x = {N}, y = 4.0, col sep = comma] {rt_wtd_nrm_ilu_cfl.csv};
	
	\addlegendentry{cfl = 20.0}
	\addplot[mark = diamond] table [x = {N}, y = 20.0, col sep = comma] {rt_wtd_nrm_ilu_cfl.csv};
	
	\end{axis}
	\end{tikzpicture}    
  \end{subfigure}  
  \caption{Repeating the experiments in Figure~\ref{fig:wtdnrm}, but with the inverse of the top left block approximated by ILU(0).  We see an increase in iteration count, and greater parameter dependence.  At moderate parameter values the increased iteration count is relatively small.}
  \label{fig:wtdnrmilu}
\end{figure}
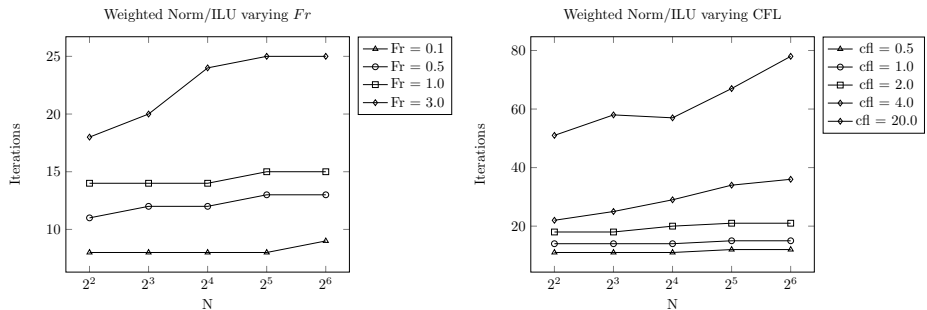

We also repeat these experiments using the decoupled preconditioner suggested in~\eqref{eq:noa} in the upper left block.
Again, we present iteration counts exactly inverting the blocks in Figure~\ref{fig:wtdnrm} and using ILU(0) in the top left block in Figures~\ref{fig:wtdnrmilu}.
Perhaps unsurprisingly, we lose some parameter robustness with respect to the Froude and CFL numbers, but our iteration counts are only about 2-4 times as large as the respective iteration counts in Figures~\ref{fig:wtdnrm} and~\ref{fig:wtdnrmilu}.
The much-reduced sparsity of the preconditioner and hence its ILU(0) factorization could compensate for that increase.

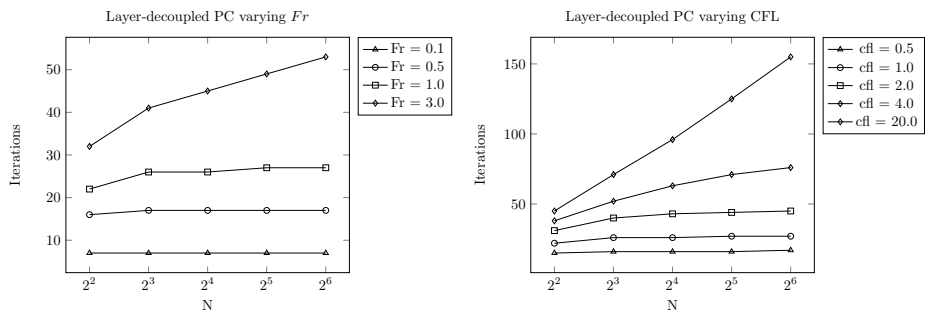
\begin{figure}
  \begin{subfigure}[c]{0.475\textwidth}
	\begin{tikzpicture}[scale=0.55]
	\begin{axis}[title = {Layer-decoupled PC varying $Fr$}, 
	xmode = log,
	log basis x={2},
	xlabel = {N}, ylabel = {Iterations},
	legend pos = outer north east]
	\addlegendentry{Fr = 0.1}
	\addplot[mark = triangle] table [x = {N}, y = 0.1, col sep = comma] {rt_noA_lu_fr.csv};
	
	\addlegendentry{Fr = 0.5}
	\addplot[mark = o] table [x = {N}, y = 0.5, col sep = comma] {rt_noA_lu_fr.csv};
	
	\addlegendentry{Fr = 1.0}
	\addplot[mark = square] table [x = {N}, y = 1.0, col sep = comma] {rt_noA_lu_fr.csv};
	
	\addlegendentry{Fr = 3.0}
	\addplot[mark = diamond] table [x = {N}, y = 3.0, col sep = comma] {rt_noA_lu_fr.csv};
	
	\end{axis}
	\end{tikzpicture}
  \end{subfigure}
  \begin{subfigure}[c]{0.475\textwidth}
	\begin{tikzpicture}[scale=0.55]
	\begin{axis}[title = {Layer-decoupled PC varying CFL}, 
	xmode = log,
	log basis x={2},
	xlabel = {N}, ylabel = {Iterations},
	legend pos = outer north east]
	\addlegendentry{cfl = 0.5}
	\addplot[mark = triangle] table [x = {N}, y = 0.5, col sep = comma] {rt_noA_lu_cfl.csv};
	
	\addlegendentry{cfl = 1.0}
	\addplot[mark = o] table [x = {N}, y = 1.0, col sep = comma] {rt_noA_lu_cfl.csv};
	
	\addlegendentry{cfl = 2.0}
	\addplot[mark = square] table [x = {N}, y = 2.0, col sep = comma] {rt_noA_lu_cfl.csv};
	
	\addlegendentry{cfl = 4.0}
	\addplot[mark = diamond] table [x = {N}, y = 4.0, col sep = comma] {rt_noA_lu_cfl.csv};
	
	\addlegendentry{cfl = 20.0}
	\addplot[mark = diamond] table [x = {N}, y = 20.0, col sep = comma] {rt_noA_lu_cfl.csv};
	
	\end{axis}
	\end{tikzpicture}
  \end{subfigure}    
  \caption{Performance of the weighted norm preconditioner using the simplied form~\eqref{eq:noa} in the top left block. Exact inversion of the blocks.}
  \label{fig:wtdnrmnoa}
\end{figure}

\begin{figure}
    \begin{subfigure}[c]{0.475\textwidth}
	\begin{tikzpicture}[scale=0.55]
	\begin{axis}[title = {Layer-decoupled PC/ILU varying $Fr$}, 
	xmode = log,
	log basis x={2},
	xlabel = {N}, ylabel = {Iterations},
	legend pos = outer north east]
	\addlegendentry{Fr = 0.1}
	\addplot[mark = triangle] table [x = {N}, y = 0.1, col sep = comma] {rt_noA_ilu_fr.csv};
	
	\addlegendentry{Fr = 0.5}
	\addplot[mark = o] table [x = {N}, y = 0.5, col sep = comma] {rt_noA_ilu_fr.csv};
	
	\addlegendentry{Fr = 1.0}
	\addplot[mark = square] table [x = {N}, y = 1.0, col sep = comma] {rt_noA_ilu_fr.csv};
	
	\addlegendentry{Fr = 3.0}
	\addplot[mark = diamond] table [x = {N}, y = 3.0, col sep = comma] {rt_noA_ilu_fr.csv};
	
	\end{axis}
	\end{tikzpicture}
  \end{subfigure}
  \begin{subfigure}[c]{0.475\textwidth}
	\begin{tikzpicture}[scale=0.55]
	\begin{axis}[title = {Layer-decoupled PC/ILU varying CFL}, 
	xmode = log,
	log basis x={2},
	xlabel = {N}, ylabel = {Iterations},
	legend pos = outer north east]
	\addlegendentry{cfl = 0.5}
	\addplot[mark = triangle] table [x = {N}, y = 0.5, col sep = comma] {rt_noA_ilu_cfl.csv};
	
	\addlegendentry{cfl = 1.0}
	\addplot[mark = o] table [x = {N}, y = 1.0, col sep = comma] {rt_noA_ilu_cfl.csv};
	
	\addlegendentry{cfl = 2.0}
	\addplot[mark = square] table [x = {N}, y = 2.0, col sep = comma] {rt_noA_ilu_cfl.csv};
	
	\addlegendentry{cfl = 4.0}
	\addplot[mark = diamond] table [x = {N}, y = 4.0, col sep = comma] {rt_noA_ilu_cfl.csv};
	
	\addlegendentry{cfl = 20.0}
	\addplot[mark = diamond] table [x = {N}, y = 20.0, col sep = comma] {rt_noA_ilu_cfl.csv};
	
	\end{axis}
	\end{tikzpicture}
  \end{subfigure}
  \caption{Performance of the decoupled preconditioner using ILU(0) for the top left block.}
  \label{fig:wtdnrmnoailu}
\end{figure}
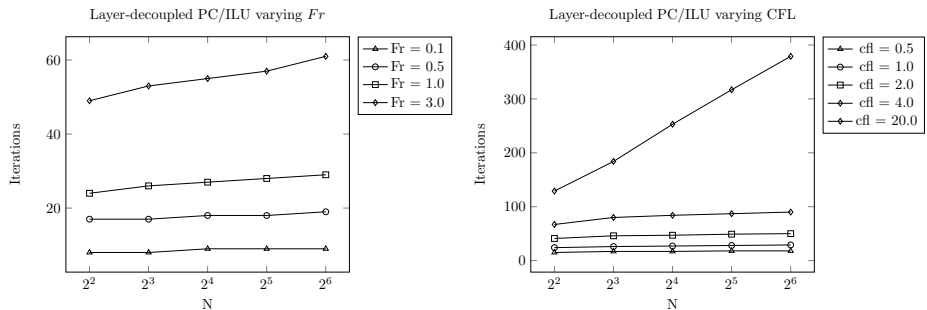

Now, we want to comment on the dependence of the preconditioners as a function of the number of layers.
For this, we fixed a $N \times N$ mesh with $N=64$ divided into triangles, fixed $Fr=\epsilon=$ and $dt = 2/N = 0.03125$ and considered the number of iterations required to solve the linear system with various preconditioners -- ILU on the original system and preconditioners~\eqref{eq:precond}and~\eqref{eq:noa}, alternately using exact inversion or an ILU approximation of the top-left block
These results are shown in Figure~\ref{fig:layers}.
None of these methods show signicant variation as we increase the number of layers.
This behavior for~\eqref{eq:precond} is not unexpected in light of
Theorems~\ref{thm:upper} and~\ref{thm:lower}, but is better than one expects given Theorem~\ref{thm:noa}.

\begin{figure}
  \begin{center}
\begin{tikzpicture}[scale=0.55]
  \begin{axis}[title = {Layer Dependence},
      xlabel = {Layers}, ylabel = {Iterations},
      legend pos = outer north east]
    
    \addlegendentry{ILU}
    \addplot[mark = triangle] table [x = {Nlayers}, y = {ilu}, col sep = comma] {layers_vs_its_2.csv};

    \addlegendentry{Wtd Norm}
    \addplot[mark = square] table [x = {Nlayers}, y = {fs_riesz_lu}, col sep = comma] {layers_vs_its_2.csv};
    
    \addlegendentry{Layer-decoupled}
    \addplot[mark = o] table [x = {Nlayers}, y = {fs_riesz_noA_lu}, col sep = comma] {layers_vs_its_2.csv};

    \addlegendentry{Wtd Norm/ILU}
    \addplot[mark = *] table [x = {Nlayers}, y = {fs_riesz_ilu}, col sep = comma] {layers_vs_its_2.csv};
    
    \addlegendentry{Layer-decoupled/ILU}
    \addplot[mark = diamond] table [x = {Nlayers}, y = {fs_riesz_noA_ilu}, col sep = comma] {layers_vs_its_2.csv};
    
  \end{axis}
\end{tikzpicture}
\end{center}
\caption{Iteration count as a function of the number of layers.  We see that the preformance of our preconditioners seems bounded as we increase the number of layers, a result better than that predicted in Theorem~\ref{thm:noa}.}
\label{fig:layers}
\end{figure}
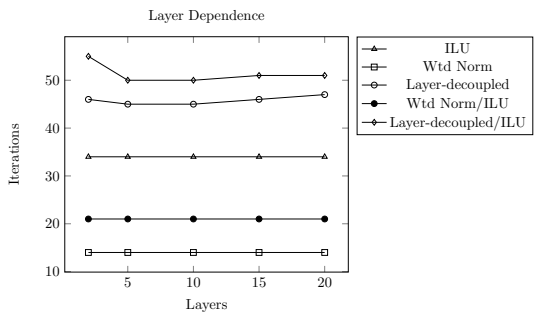

\section{Conclusions and future work}
We have presented a new tide model based on the linearized rotating shallow water equations, but with several layers stratified by density.
A mixed finite element discretization similar to that for single-layer models~\cite{CoKi} gives rise to a large system of equations, with additional complexity arising from the all-to-all coupling between the layers.
We have presented and analyzed weighted-norm preconditioners that are robust with respect to most of the physical and discretization parameters.
For typical parameter values, additional approximations such as neglecting inter-layer coupling and approximating inverses of matrix blocks with incomplete factorizations may result in highly practical methods.

Future directions for this work would include careful energy-type estimates that sharply describe the physical damping in the system.
These would inform \emph{a priori} estimates like we have previously  derived in the single layer case.
Additionally, adapting such energy estimates to the fully discrete case, as well as studying the systems arising from higher-order temporal methods, present further challenges.

\bibliographystyle{siamplain}
\bibliography{bib}
\end{document}